\documentclass[11pt]{article}
\usepackage[letterpaper, margin=1in]{geometry}

\usepackage[utf8]{inputenc}
\usepackage{lmodern}

\usepackage{amsmath}
\usepackage{amsfonts}
\usepackage{amsthm}
\usepackage{mathtools}
\usepackage{paralist}
\usepackage{xcolor}
\usepackage{hyperref}
\usepackage[disable]{todonotes}
\usepackage{algpseudocode}
\usepackage{algorithm}
\usepackage{bm}
\usepackage[affil-it]{authblk}
\usepackage{tikz}
\usetikzlibrary{tikzmark, shapes.misc}

\usepackage{biblatex}
\title{Gadget construction and structural convergence\thanks{This paper is part of a project that has received funding from the European Research Council (ERC) under the European Union’s Horizon 2020 research and innovation programme (grant agreement No 810115 – Dynasnet).}}
\author[1,2]{David Hartman}
\author[1]{Tom\'{a}\v{s} Hons}
\author[1]{Jaroslav Ne\v{s}et\v{r}il}
\affil[1]{\small Computer Science Institute, Faculty of Mathematics and Physics, Charles University, Prague, Czech Republic}
\affil[2]{\small Institute of Computer Science of the Czech Academy of Sciences, Prague, Czech Republic}
\date{}

\theoremstyle{plain}
\newtheorem{theorem}{Theorem}[section]
\newtheorem{corollary}[theorem]{Corollary}
\newtheorem{lemma}[theorem]{Lemma}
\newtheorem{proposition}[theorem]{Proposition}
\newtheorem*{question}{Question}

\theoremstyle{definition}
\newtheorem{definition}{Definition}
\newtheorem{example}{Example}

\theoremstyle{remark}
\newtheorem*{remark}{Remark}
\newtheorem{fact}{Fact}

\newcommand{\bigland}{\bigwedge}

\DeclareMathOperator\dist{dist}
\DeclareMathOperator\Gad{Gad}
\DeclareMathOperator\Eq{Eq}

\DeclareMathOperator\qranksym{qrank}
\newcommand{\qrank}[1]{{\qranksym({#1})}}

\DeclareMathOperator\aritysym{ar}
\newcommand{\arity}[1]{{\aritysym({#1})}}

\newcommand{\N}{\mathbb{N}}

\DeclareMathOperator\FO{FO}
\DeclareMathOperator\QF{QF}
\DeclareMathOperator{\EF}{EF}
\newcommand{\FOloc}[1]{\FO_{#1}^\textnormal{loc}}
\newcommand{\FOcloc}[1]{\FO_{#1}^{\textnormal{c-loc}}}
\newcommand{\stonepar}[2]{\langle #1, #2 \rangle}

\newcommand{\str}[1]{\mathbf{#1}}
\newcommand{\strseq}[1]{{\boldsymbol{\mathsf{#1}}}}
\newcommand{\seq}[1]{{\mathsf{#1}}}

\newcommand{\Los}{\L{}o\'{s}}
\newcommand{\maxarity}[1]{\Delta_{#1}}

\newcommand{\eps}{\varepsilon}

\newcommand{\tpl}[1]{{\bm{#1}}}

\DeclareMathOperator\Int{Int}
\DeclareMathOperator\Ext{Ext}

\newcommand{\nonroot}{\tikzmarknode[strike out,draw]{1}{\bullet}}

\DeclareMathOperator{\liftsym}{L}
\newcommand{\lift}[2]{\liftsym_{#1}({#2})}

\DeclareMathOperator{\ellim}{el-lim}

\begin{document}

\maketitle

\begin{abstract}
Ne\v{s}et\v{r}il and Ossona de Mendez recently proposed a new definition of graph convergence called structural convergence.
The structural convergence framework is based on the probability of satisfaction of logical formulas from a fixed fragment of first-order formulas.
The flexibility of choosing the fragment allows to unify the classical notions of convergence for sparse and dense graphs.
Since the field is relatively young, the range of examples of convergent sequences is limited and only a few methods of construction are known.
Our aim is to extend the variety of constructions by considering the gadget construction.
We show that, when restricting to the set of sentences, the application of gadget construction on elementarily convergent sequences yields an elementarily convergent sequence.
On the other hand, we show counterexamples witnessing that a generalization to the full first-order convergence is not possible without additional assumptions.
We give several different sufficient conditions to ensure the full convergence.
One of them states that the resulting sequence is first-order convergent if the replaced edges are dense in the original sequence of structures.
\end{abstract}

\tableofcontents

\section{Introduction}\label{sec:introduction}

\todo{Consolidate labels (once).}

The area of asymptotic properties of graphs and relational structures was recently enriched by study of convergence and limit objects.
Several types of convergences were studied mostly based on counting of homomorphisms, cut-metric convergence, and local convergence.
The area is using analysis, probability and graph theory as the main tools.
See~\cite{large_networks}\cite{counting_graph_homomorphism}\cite{limits_of_dense_graph_sequences}\cite{benjamini_schramm}\cite{elek}\cite{elek_tardos}.

In an effort to consolidate this spectrum and to find a useful type of convergence for sparse structures, many different approaches were unified in a common framework called \emph{structural convergence} \cite{unified_approach}\cite{clustering}.
It is this type of convergence, which is based on model theory and combinatorics, that is the subject of this paper.

Gadget construction (also called replacement or indicator construction) is a natural method for hierarchical assembly of graphs and other structures with applications in algebraic graph theory \cite{frucht}\cite{hedrlin_pultr}\cite{hell_nesetril}, complexity \cite{garey_johnson}\cite{feder_vardi}, and category theory \cite{pultr_trnkova}\cite{nesetril_ossona_de_mendez_categories}\cite{hubicka_nesetril}.

The goal of this paper is to treat gadget construction as a vital tool for the area of structural convergence.
We examine the effects of gadget construction $*$ on $X$-convergent sequences with a particular focus on the following question:
if $(\str{A}_n)_{n \in \N}$ is an $X$-convergent sequence of base structures and $(\str{G}_n)_{n \in \N}$ an $X$-convergent sequence of gadgets, is the sequence $(\str{A}_n * \str{G}_n)_{n \in \N}$ of results of the gadget construction $X$-convergent as well?
In such a case, we say that the $X$-convergence is preserved by the gadget construction.
This is studied and characterized in the presented paper.

We focus separately on elementary and local convergence, whose combination implies full first-order convergence (see Section~\ref{ssec:structural_convergence}).
We show that gadget construction $*$ is continuous when considered as a mapping between spaces of structures with metrics based on elementary equivalence (Theorem~\ref{thm:continuity_of_gadget_construction}).
It follows that gadget construction preserves elementary convergence (Corollary~\ref{cor:preservation_of_elementary_convergence}).
This is not true for local convergence (Examples~\ref{ex:basic_fluctuation}, \ref{ex:counterexample_subdivision}, and \ref{ex:magnification}) and additional assumptions are necessary (Theorem~\ref{thm:general_local_convergence}).
In particular, local convergence is preserved if the replaced edges are \emph{dense} in the sequence of base structures in the sense that the limit density (proportion of present vs. possible edges) is positive (Corollary~\ref{cor:fo_dense_edges}).
Moreover, under some additional assumptions, we prove that the given sufficient conditions for local convergence are optimal (Theorem~\ref{thm:inverse_for_construction_language}).
Nevertheless, we show that both for elementary and local convergence the conditions on the sequence of base structures can be relaxed provided that the gadgets are stretching (Theorem~\ref{thm:preservation_of_elementary_convergence_with_fragmentation} and Corollary~\ref{cor:loc_fragmented_structures}).

We present two simple applications of gadget construction.
We show that an arbitrary sequence of structures is $\FO$-convergent if and only if a modified sequence of very sparse structures is $\FO$-convergent (Proposition~\ref{prop:reduction_to_sparse}).
Moreover, we give a short probabilistic construction of a sequence of graphs which is almost surely $\FO_{k-1}$-convergent but not $\FO_k$-convergent for any fixed $k \geq 2$ (Example~\ref{ex:bounded_convergence}).

An essential part of this paper is based on the thesis of the second author \cite{diploma_thesis}.

\paragraph*{Organization}
In Section~\ref{sec:prelim}, we briefly introduce all necessary notions and used notation.
Section~\ref{sec:elementary_convergence} contains our results on elementary convergence.
In Section~\ref{sec:obstacles}, we show that local convergence is not always preserved and identify the main obstacles.
Section~\ref{sec:positive_cases} is devoted to positive results on preservation of local convergence, which are complemented by inverse theorems in Section~\ref{sec:inverse_theorems}.
In Section~\ref{sec:applications}, we give two simple applications of the developed theory.
The last section contains concluding remarks and open problems.
\section{Preliminaries}\label{sec:prelim}

We use $\N = \{1, 2, \dots\}, \N_0 = \{0\} \cup \N$ and $[n] = \{1, \dots, n\}, [n]_0 = \{0\} \cup [n]$.

Our languages are relational with equality and possibly with constants.
All arities are finite.
Generic languages are denoted by the Greek letter $\lambda$ while languages related to gadget construction will be denoted by $L$, possibly with various subscripts or superscripts.
The arity of a symbol $S \in \lambda$ is written as $\arity{S}$.
We use $\maxarity{\lambda}$ to denote $\max_{S \in \lambda} \arity{S}$.
The set of all first-order formulas of the language $\lambda$ is written as $\FO(\lambda)$ while $\FO_p(\lambda)$ is used for the formulas with $p$ free variables.
In particular, $\FO_0(\lambda)$ stands for the set of $\lambda$-sentences.
We often omit the explicit mention of the language and write $\FO, \FO_p$, etc. instead.

A structure $\str{A}$ over a language $\lambda$, a $\lambda$-structure, is given by its vertex set and realizations of the symbols from $\lambda$.
The structures are denoted by boldface letters $\str{A}, \str{B}$, etc., the vertex set of $\str{A}$ is $V(\str{A})$ and the realization of a symbol $S \in \lambda$ in the structure $\str{A}$ is $S^\str{A} \subseteq V(\str{A})^{\arity{S}}$.
The elements of $S^\str{A}$ are called $S$-edges, or simply edges.
Our structures are finite unless mentioned otherwise.

Let $\lambda'$ be an extension of $\lambda$ by some symbols (this expression implicitly assumes that the extending symbols are not in $\lambda$).
Then a $\lambda'$-structure $\str{A}$ can be regarded as a $\lambda$-structure $\str{B}$ by forgetting the realizations of excessive symbols outside $\lambda$.
The structure $\str{B}$, also denoted by $\str{A}|_\lambda$, is called the $\lambda$-shadow of $\str{A}$, while $\str{A}$ is a $\lambda'$-lift of $\str{B}$.
Naturally, $\str{B}$ can be viewed as a $\lambda'$-structure with empty realization $S^\str{A}$ for every $S \in \lambda' \setminus \lambda$ provided that all the additional symbols are relational.

The distance of vertices $u$ and $v$ in the structure $\str{A}$, written as $\dist_\str{A}(u,v)$, is defined as their distance in the Gaifman graph of the structure $\str{A}$.
We use the usual convention that a pair of unreachable vertices has distance $+\infty$.

The tuples (e.g. elements of $S^\str{A}$ or free variables of a formula) are named by boldface lowercase letters $\tpl{a}, \tpl{b}, \tpl{x}$, etc. and we refer to their elements using indices, e.g. $a_1, b_i, x_n$.
Occasionally, after an explicit mention, we use the function notation regarding an $n$-tuple as a function on $[n]$.
The $i$-th element of a tuple $\tpl{a}$ is then referred to as $\tpl{a}(i)$.
A tuple of length $p$ is called a $p$-tuple and the length of a tuple $\tpl{a}$ is denoted by $|\tpl{a}|$.

Let $X$ be a subset of $V(\str{A})$.
The set of vertices in the distance at most $r$ from $X$ in $\str{A}$, the $r$-neighborhood of $X$, is denoted by $N_\str{A}^r(X)$.
We write $\partial_\str{A} X$ for the boundary of $X$ in $\str{A}$, which is the set $N_\str{A}(X) \setminus X$.
The uniform measure on $V(\str{A})$ is denoted by $\nu_\str{A}$, i.e. $\nu_\str{A}(X)$ is the relative size of $X$ within $\str{A}$.
If $X$ does not contain any constants, then $\str{A} \setminus X$ stands for the substructure of $\str{A}$ induced by $V(\str{A}) \setminus X$.

Let $\str{A}$ be a $\lambda$-structure with a vertex $a$.
Then $(\str{A}, a)$ stands for the structure $\str{A}$ rooted at $a$.
Formally, we expand $\lambda$ by a new constant $c$ to the language $\lambda^+$ and $(\str{A}, a)$ is a $\lambda^+$-lift of $\str{A}$ with $c^{(\str{A},a)} = a$.
For a tuple of vertices $\tpl{a}$, the structure $(\str{A}, \tpl{a})$ is obtained by repeated rooting of vertices from $\tpl{a}$.
Let $\str{B}$ be a structure with roots and $r \in \N_0$.
By $\str{B}^r$ we denote the substructure of $\str{B}$ induced by the $r$-neighborhood of roots.

We use boldface sans-serif letters $\strseq{A}$ as a shorthand for sequences of structures.
That is, the sequence $(\str{A}_n)_{n \in \N}$ is denoted by $\strseq{A}$.
The lightface letters $\seq{X}$ are for sequences of sets.
If $f$ is a operator on structures or sets, then $f(\strseq{A})$ is the sequence $(f(\str{A}_n))_{n \in \N}$; similarly with more operands and relations.
For example, if $\seq{X} \subseteq V(\strseq{A})$, i.e. $\forall n: X_n \subseteq V(\str{A}_n)$, then $\nu_\strseq{A}(\seq{X})$ is the sequence $(\nu_{\str{A}_n}(X_n))_{n \in \N}$.
For a property $P$ of structures (or sets), we say that $\strseq{A}$ \emph{eventually} satisfies $P$ if there is $n_0 \in \N$ such that all $\str{A}_n$ for $n \geq n_0$ satisfy $P$.
For example, $\seq{X}$ eventually does not contain a root of $\strseq{A}$ if each $X_n$ from a certain index on does not contain a root.

In the rest of this section, we give basic definitions regarding gadget construction and structural convergence.
We also recall the Ehrenfeucht-Fra\"{i}ss\'{e} games, which one of our main tools.

\subsection{Gadget construction}\label{ssec:gadget_construction}

Gadget construction is an operation that takes two structures $\str{A}$ and $\str{G}$ and replaces each edge of a particular relation of $\str{A}$ by a copy of $\str{G}$ identifying a specified tuple of vertices of $\str{G}$ with the vertices of the replaced edge.

Throughout the paper, we fix a purely relational language $L$.

\begin{definition}[Base structure and gadget]\label{def:base_structure_and_gadget}
    Let $L_R$ be the language $L$ extended by a symbol $R$.
    An $L_R$-structure $\str{A}$ is called a \emph{base} structure.
    
    Let $L_G$ be the language $L$ extended by constants $z_1, \dots, z_{\arity{R}}$.
    An $L_G$-structure $\str{G}$ with pairwise distinct vertices $z_i^\str{G}$ is called a \emph{gadget}.
\end{definition}

\begin{definition}[Gadget construction]\label{def:gadget_construction}
    Let $\str{A}$ be a base structure and $\str{G}$ a gadget.
    By $\str{A} * \str{G}$ we denote the $L$-structure that is the result of gadget construction applied on the base structure $\str{A}$ and the gadget $\str{G}$.
    We define
    \[
        V(\str{A} * \str{G}) = \big(V(\str{A}) \cup (R^\str{A} \times V(\str{G})) \big) /_{\sim}
        ,
    \]
    where $\sim$ is the equivalence generated by the pairs $(a,(\tpl{e},v))$ satisfying that there is $i \in [\arity{R}]$ such that $a = e_i$ and $v = z_i^\str{G}$.
    Denote by $[x]$ the $\sim$-class of $x$.
    For a symbol $S \in L$ of arity $s$, we set
    \begin{align*}
          S^{\str{A} * \str{G}} 
        = 
             &\{([a_1], \dots, [a_s]) : (a_1, \dots, a_s) \in S^\str{A} \} \\
        \cup &\{([(\tpl{e},v_1)], \dots, [(\tpl{e},v_s)]) : \tpl{e} \in R^\str{A}, (v_1, \dots, v_s) \in S^\str{G}\}
        .
    \end{align*}
\end{definition}

\missingfigure{Example of gadget construction.}

We can view gadget construction as a replacement operation.
Each $R$-edge $\tpl{e} \in R^\str{A}$ is replaced in $\str{A}*\str{G}$ by a copy $\str{G}^\tpl{e}$ of $\str{G}$, identifying the vertices of $\tpl{e}$ with the roots of $\str{G}^\tpl{e}$.
All the copies $\str{G}^\tpl{e}$ are vertex-disjoint, except possibly for their roots.
We denote by $\iota_\tpl{e}$ the natural mapping from $\str{G}^\tpl{e}$ to $\str{G}$.
We remark that although we call the vertices $\iota_\tpl{e}^{-1}(\tpl{z}^\str{G})$ the roots of $\str{G}^\tpl{e}$, they do not interpret the constants $\tpl{z}$ in $\str{A}*\str{G}$, which is merely an $L$-structure.
Moreover, $\iota_\tpl{e}^{-1}$ is an embedding of $\str{G}|_L$ but not necessarily an isomorphism: extra edges (originating from $\str{A}$) may span the roots of $\str{G}^\tpl{e}$.

A vertex of $\str{A}*\str{G}$ is \emph{internal} if it contains (as an equivalence class) a vertex of $\str{A}$.
The remaining vertices of $\str{A}*\str{G}$, i.e. non-root vertices in some $\str{G}^\tpl{e}$, are \emph{external}.
We usually identify a vertex $a$ of $\str{A}$ with the vertex $[a]$ of $\str{A}*\str{G}$.
For an external vertex $a$ in $\str{G}^\tpl{e}$, let $\rho(a)$ denote the tuple $\tpl{e}$.

Throughout the paper, we represent the vertices of $\str{A} * \str{G}$ in the structures $\str{A}$ and $\str{G}$ with the aim to transfer convergence from $\strseq{A}$ and $\strseq{G}$ to $\strseq{A} * \strseq{G}$.
An internal vertex $a$ can be directly considered as a vertex of $\str{A}$ while an external vertex $a$ is uniquely given by the $R$-edge $\rho(a) \in R^\str{A}$ and the vertex $\iota_{\rho(a)}(a) \in V(\str{G})$.

\begin{remark}
    We emphasize that all edges, as tuples, have their implicit orientation (ordering of vertices).
    Therefore, if all symbols are binary, we are speaking about (colored) digraphs.
    If $\str{A}$ is a symmetric digraph, each pair of neighbors gets two copies of $\str{G}$ when constructing $\str{A}*\str{G}$ as there is one arc in each direction.
    It is possible to extend gadget construction to undirected graphs (putting only one copy of $\str{G}$ between neighbors) provided that the gadget itself is undirected in the sense that it admits an automorphism that swaps the roots.
    Similarly, we can apply gadget construction to hypergraphs.
    The techniques we develop here work equally well in the undirected setting.
    In our examples, we prefer to use undirected structures.
\end{remark}

\subsection{Structural convergence}\label{ssec:structural_convergence}

We briefly recall the basic definitions related to the structural convergence framework, see \cite{unified_approach} for a detailed exposition.

For a formula $\phi \in \FO_p(\lambda)$, $p \geq 1$, and a $\lambda$-structure $\str{A}$, the \emph{Stone pairing} $\stonepar{\phi}{\str{A}}$ is the probability that we have $\str{A} \models \phi(\tpl{a})$ for a uniformly chosen $p$-tuple $\tpl{a}$ of vertices of $\str{A}$.
In the special case of sentences, we set $\stonepar{\phi}{\str{A}} = 1$ if $\str{A} \models \phi$, and $\stonepar{\phi}{\str{A}} = 0$ otherwise.
Let $X$ be a subset of $\FO(\lambda)$.
A sequence $\strseq{A}$ of $\lambda$-structures is \emph{$X$-convergent} if the sequence $\stonepar{\phi}{\strseq{A}}$, i.e. $(\stonepar{\phi}{\str{A}_n})_{n \in \N}$, converges for each $\phi \in X$.
 
Apart from $\FO$-convergence, the important cases include $\FO_0$-convergence, also called \emph{elementary convergence}, and $\FOloc{}$-convergence, \emph{local convergence}, where $\FOloc{}$ is the set of \emph{local} formulas.
A formula is \emph{$r$-local} if its satisfaction depends only on the $r$-neighborhood of its free variables and \emph{local} if there is $r \in \N_0$ such that it is $r$-local.
The Gaifman theorem states that any $\phi \in \FO$ can be expressed as a boolean combination of sentences and local formulas~\cite{gaifman}, which implies that $\strseq{A}$ is $\FO$-convergent if and only if it is both elementarily convergent and local convergent \cite[Theorem~2.23]{unified_approach}.

On top of the local formulas, we define \emph{constant-local} formulas.
A formula is $r$-constant-local if its satisfaction depends only on the $r$-neighborhood the free variables \emph{and} constants.
The set of constant-local formulas is denoted by $\FOcloc{}$.
If the language is purely relational, it holds $\FOcloc{} = \FOloc{}$ and $\FOcloc{0} = \emptyset$.
In the general case, however, we have $\FOloc{} \subseteq \FOcloc{}$ and $\emptyset = \FOloc{0} \subseteq \FOcloc{0} \subseteq \FO_0$.
It is easy to see that a constant-local formula can be written as a boolean combination of constant-local sentences and local formulas.
(For the case of a single variable, single constant, and an $r$-constant-local formula: distinguish whether the variable and the root are in distance at most $2r$; this is possible by a $2r$-local formula.
If they are, a $3r$-local formula suffices.
Otherwise, the satisfaction depends independently on an $r$-local formula and $r$-constant-local sentence.)

Two $\lambda$-structures $\str{A}$ and $\str{B}$ (of arbitrary cardinality) are \emph{$k$-elementarily equivalent}, $\str{A} \equiv_k \str{B}$, if $\str{A} \models \phi \Leftrightarrow \str{B} \models \phi$ for each sentence $\phi$ of quantifier rank at most $k$.
The quantifier rank of a formula $\phi$, $\qrank{\phi}$, is the maximal depth of nesting of quantifiers in the structural tree of $\phi$.
The structures $\str{A}$ and $\str{B}$ are \emph{elementarily equivalent} if $\str{A} \equiv_k \str{B}$ for each $k \in \N$.
It is a well-known fact that $k$-elementary equivalence is an equivalence of finite index.
Moreover, each class $\mathcal{C}$ of $\equiv_k$ can be described by a formula $\phi$ with $\qrank{\phi} = k$ satisfying that $\str{A} \in \mathcal{C}$ if and only if $\str{A} \models \phi$.
This also applies to $\equiv_k$ on structures $(\str{A}, \tpl{a})$ with $p$ roots, where the roots can be supplied to the formula as arguments.
That is, $(\str{A}, \tpl{a}) \in \mathcal{C}$ if and only if $\str{A} \models \phi(\tpl{a})$ \cite{hodges}.


Elementary convergence coincides with a metric-convergence in the space $\lambda$-structures $\mathcal{A}$ (of arbitrary cardinality).
The distance $\rho(\str{A}, \str{B})$ of $\str{A}$ and $\str{B}$ is defined as $\inf \{ 2^{-k} : \str{A} \equiv_k \str{B} \}$.
The function $\rho$ is a pseudo-ultrametric on the set $\mathcal{A}$ and the space $(\mathcal{A}, \rho)$ is compact \cite[Proposition~2.18]{unified_approach}.
It follows from the definition that a sequence $\strseq{A}$ is elementarily convergent if and only if it is $\rho$-convergent.
Therefore, an elementarily convergent sequence $\strseq{A}$ has a $\rho$-limit in the space $(\mathcal{A}, \rho)$.
More precisely, there is a set of (possibly non-isomorphic, but elementarily equivalent) limits as $\rho$ is only a pseudo-metric.
Any such a limit structure is called an \emph{elementary limit} of $\strseq{A}$ and we denote it by $\ellim \strseq{A}$.
As an example of a particular elementary limit of $\strseq{A}$ serves the ultraproduct $\prod_{n \in \N} \str{A}_n / \mathcal{U}$, or rather $\prod_\N \strseq{A} / \mathcal{U}$, where $\mathcal{U}$ is a non-principal ultrafilter on $\N$ \cite{chang_keisler}.

\subsection{Ehrenfeucht-Fra\"{i}ss\'{e} games}\label{ssec:ef_games}

Here we recall the Ehrenfeucht-Fra\"{i}ss\'{e} games \cite{fraisse}\cite{ehrenfeucht}, which is an important tool in model theory, particularly in finite model theory.

The $k$-round Ehrenfeucht-Fra\"{i}ss\'{e} game ($\EF$-game for short) on $\lambda$-structures $\str{A}$ and $\str{B}$, denoted by $\EF_k(\str{A}; \str{B})$, is a perfect information game of two players \emph{Spoiler} and \emph{Duplicator}.
The game lasts for $k$ rounds, each of which consists of the following two steps: Spoiler chooses one of the structures $\str{A}$ or $\str{B}$ and picks a vertex from it.
Then Duplicator picks a vertex from the other structure.
We denote the vertex picked in the $i$-th round from the structure $\str{A}$ (resp. $\str{B}$) by $a_i$ (resp. $b_i$).
Duplicator wins if the structures $(\str{A}, \tpl{a})^0$ and $(\str{B}, \tpl{b})^0$ are isomorphic and loses otherwise.
Note that the only candidate for the isomorphism maps $a_i \mapsto b_i$ for each $i$ and $c^\str{A} \mapsto c^\str{B}$ for each constant $c$.

We write $\EF_k(\str{A}, \tpl{a}; \str{B}, \tpl{b})$ to emphasize that the tuples $\tpl{a}$ and $\tpl{b}$ were already selected while $k$ rounds remain to be played.
If a player has a winning strategy in a particular game, we say that the player wins the game.

The $\EF$-games are linked to the notion of elementary equivalence by the theorem of Fra\"{i}ss\'{e}: for structures $\str{A}$ and $\str{B}$, we have $\str{A} \equiv_k \str{B}$ if and only if Duplicator wins $\EF_k(\str{A};\str{B})$.
More generally, let $\tpl{a}$ and $\tpl{b}$ be $p$-tuples of vertices from $\str{A}$ and $\str{B}$, respectively, then Duplicator wins $\EF_k(\str{A},\tpl{a}; \str{B}, \tpl{b})$ if and only if $(\str{A},\tpl{a}) \equiv_k (\str{B},\tpl{b})$, i.e. $\str{A} \models \phi(\tpl{a}) \Leftrightarrow \str{B} \models \phi(\tpl{b})$ for every $\phi \in \FO_p$ with $\qrank{\phi} \leq k$ \cite{fraisse}.
\section{Elementary convergence}\label{sec:elementary_convergence}

In this section, we focus on sufficient conditions for elementary convergence of the sequence $\strseq{A} * \strseq{G}$.
In the first part, we use Ehrenfeucht-Fra\"{i}ss\'{e} games to prove that gadget construction $*$ is a continuous function between spaces of structures with respect to natural metrics.
As a consequence, we obtain that the sequence $\strseq{A} * \strseq{G}$ is elementarily convergent if both $\strseq{A}$ and $\strseq{G}$ are elementarily convergent.
In the second part, we introduce the idea of \emph{fragmentation}, which allows us to state a condition ensuring the elementary convergence of $\strseq{A} * \strseq{G}$ without requiring the elementary convergence of $\strseq{A}$.

Results of this section extend to structures of arbitrary cardinality.

\subsection{Continuity of gadget construction}\label{ssec:continuity_of_gadget_construction}

Let $(\mathcal{B},\rho_\mathcal{B}), (\mathcal{G},\rho_\mathcal{G})$, and $(\mathcal{R}, \rho_\mathcal{R})$ be the respective spaces of all base structures, gadgets, and structures resulting from gadget construction with the pseudo-metrics defined by elementary equivalence.
We consider the product space $(\mathcal{B},\rho_\mathcal{B}) \times (\mathcal{G}, \rho_\mathcal{G})$ with the distance between pairs $(\str{A}_1, \str{G}_1)$ and $(\str{A}_2, \str{G}_2)$ defined as $\max\{ \rho_\mathcal{B}(\str{A}_1, \str{A}_2), \rho_\mathcal{G}(\str{G}_1, \str{G}_2) \}$, which yields a compact pseudo-ultrametric space.
We claim that gadget construction is a continuous function with respect to these metrics.

\begin{theorem}\label{thm:continuity_of_gadget_construction}
    Gadget construction $*: (\mathcal{B}, \rho_\mathcal{B}) \times (\mathcal{G}, \rho_\mathcal{G}) \to (\mathcal{R}, \rho_\mathcal{R})$ is continuous.
\end{theorem}

As an immediate corollary, using basic properties of the continuous functions, we obtain that gadget construction preserves elementary convergence and commutes with taking the elementary limit.

\begin{corollary}\label{cor:preservation_of_elementary_convergence}
    Let $\strseq{A}$ be an elementarily convergent sequence of base structures, $\strseq{G}$ an elementarily convergent sequence of gadgets.
    Then the sequence $\strseq{A} * \strseq{G}$ is elementarily convergent and we have
    \[
        \ellim (\strseq{A}*\strseq{G}) \equiv (\ellim \strseq{A})*(\ellim \strseq{G})
        .
    \]
\end{corollary}

As noted before, the limit of $\strseq{A}$ can be expressed by an ultraproduct $\left( \prod \strseq{A} / \mathcal{U} \right)$ over a non-principal ultrafilter.
Hence we have that
\[
    \prod (\strseq{A} * \strseq{G}) / \mathcal{U} 
    \equiv
    \left( \prod \strseq{A} / \mathcal{U} \right)
    * 
    \left( \prod \strseq{G} / \mathcal{U} \right)
    ,
\]
where $\strseq{A}$ and $\strseq{G}$ are elementarily convergent sequences.

A routine use of the \Los{} theorem~\cite{los} (the fundamental theorem of ultraproducts) shows that a similar statement holds for general indexed families.
That is, if $I$ is an index set with an ultrafilter $\mathcal{U}$, then for arbitrary families $(\str{A}_i)_{i \in I}$ of base structures and $(\str{G}_i)_{i \in I}$ of gadgets we have
\[
    \prod_I (\str{A}_i * \str{G}_i) / \mathcal{U} 
    \equiv 
    \left( \prod_I \str{A}_i / \mathcal{U} \right) * \left( \prod_I \str{G}_i / \mathcal{U} \right)
    .
\]

\subsubsection{Proof of Theorem~\ref{thm:continuity_of_gadget_construction}}\label{sssec:proof_ef_games}

Here we prove Theorem~\ref{thm:continuity_of_gadget_construction} using Ehrenfeucht-Fra\"{i}ss\'{e} games.
A refinement of the idea is later used in Section~\ref{sssec:proof_of_fragmentation}.

Recall that an internal vertex $a$ of $\str{A} * \str{G}$ correspond to a vertex of $\str{A}$ while an external vertex $a$ lies in a unique copy $\str{G}^\tpl{e}$, where $\tpl{e} = \rho(a)$, and corresponds to the non-root vertex $\iota_\tpl{e}(a)$ of $\str{G}$.

The following lemma gives a particular bound on the continuity of $*$.

\begin{lemma}\label{lem:continuity_bound_via_ef_games}
    Let $\str{A}_1, \str{A}_2$ be base structures and $\str{G}_1, \str{G}_2$ be gadgets satisfying
    \begin{equation*}
        \begin{gathered} 
            \str{A}_1 \equiv_{k \cdot \arity{R}} \str{A}_2
            , \\
            \str{G}_1 \equiv_{k} \str{G}_2
            .
        \end{gathered}
    \end{equation*}
    Then we have
    \[
        \str{A}_1 * \str{G}_1 \equiv_k \str{A}_2 * \str{G}_2
        .
    \]
\end{lemma}
\begin{proof}
    We give an algorithm showing that Duplicator's winning strategy in the game $H = \EF_k(\str{A}_1*\str{G}_1; \str{A}_2*\str{G}_2)$ can be compiled from winning strategies in the games $H_\str{A} = \EF_{k \cdot \arity{R}}(\str{A}_1; \str{A}_2)$ and $H_\str{G} = \EF_k(\str{G}_1; \str{G}_2)$ and prove its correctness.
    In each round of $H$, the deduction of Duplicator's response follows Algorithm~\ref{alg:duplicator_strategy_simple}.
    
\begin{algorithm}
    \caption{Duplicator's response in one round of $H$}
    \label{alg:duplicator_strategy_simple}
    \begin{algorithmic}[1]
        \State $S, D \gets$ indices of the Spoiler's and Duplicator's structure in this round
        \State $u \gets$ vertex chosen by Spoiler from $\str{A}_S * \str{G}_S$ 
        \If{$u$ is internal}
            \State Let Spoiler pick $\str{A}_S$ and the vertex $u$ in $H_\str{A}$
            \State $v \gets$ Duplicator's response in $H_\str{A}$
        \Else
            \State $\tpl{e} \gets \rho(u)$
            \State Let Spoiler pick $\str{A}_S$ and all the vertices of $\tpl{e}$ in $H_\str{A}$
                \label{step:pick_r_edge}
            \State $\tpl{f} \gets$ Duplicator's response in $H_\str{A}$
                \label{step:selecting_corresponding_r_edge}
            \State Let Spoiler pick $\str{G}_S$ and the vertex $\iota_\tpl{e}(u)$ in $H_\str{G}$
            \State $v' \gets$ Duplicator's response in $H_\str{G}$
            \State $v \gets \iota_\tpl{f}^{-1}(v')$
        \EndIf
        \State Vertex $v$ is the Duplicator's response
    \end{algorithmic}
\end{algorithm}

    We prove that this is a winning strategy.
    First observe that we do not exceed the number rounds of the game $H_\str{A}$ nor $H_\str{G}$ and that the tuple $\tpl{f}$ obtained in Step~\ref{step:selecting_corresponding_r_edge} is an $R$-edge, which makes the vertex $v$ well defined. 

    Let $\tpl{a}$ and $\tpl{b}$ be the $t$-tuples chosen from $\str{A}_1*\str{G}_1$ and $\str{A}_2*\str{G}_2$ after $t$ rounds.
    We want $\alpha: a_i \mapsto b_i$ to be an isomorphism between the substructures induced by $\tpl{a}$ and $\tpl{b}$.
    Suppose that there is an edge $\tpl{u} \in S^{\str{A}_1*\str{G}_1}$ with $u_j = a_{i_j}$ for some indices $i_j \in [t]$ for all $j \in [\arity{S}]$.
    We need to show that $\alpha(\tpl{u}) = \tpl{v} \in S^{\str{A}_2*\str{G}_2}$ (the converse direction is by symmetry).
        
    We distinguish whether the $S$-edge spanning $\tpl{u}$ originated from $\str{A}_1$ or from $\str{G}_1$.
    If it is from $\str{A}_1$, all the vertices of $\tpl{u}$ are internal and we have $\tpl{u} \in S^{\str{A}_1}$.
    Let $\beta: V(\str{A}_1) \to V(\str{A}_2)$ be the partial isomorphism of the picked vertices in the game $H_\str{A}$ (i.e. the domain of $\beta$ contains only the picked vertices).
    We have $\beta(\tpl{u}) \in S^{\str{A}_2}$ as $\tpl{u}$ belongs to the domain of the partial isomorphism $\beta$.
    Moreover, $\tpl{v} = \beta(\tpl{u})$ by Algorithm~\ref{alg:duplicator_strategy_simple}.
    Therefore, $\tpl{v} \in S^{\str{A}_2*\str{G}_2}$.
    
    Now, suppose that the $S$-edge originated from $\str{G}_1$, i.e. there is an $R$-edge $\tpl{e}$ in $\str{A}_1$ such that all the vertices of $\tpl{u}$ belong to $\str{G}^\tpl{e}_1$; in particular, the internal vertices of $\tpl{u}$ belong to $\tpl{e}$.
    
    Observe that there is an edge $\tpl{f} \in R^{\str{A}_2}$ with $f_j = v_i$ if and only if $e_j = u_i$ for all $i,j$ such that all the vertices of $\tpl{v}$ belong to $\str{G}^\tpl{f}_2$.
    This is because Spoiler always has enough rounds in $H_\str{A}$ to select all the remaining vertices of $\tpl{e}$ and Duplicator needs to be able to mirror such a selection.
    However, $\tpl{f}$ need not to be uniquely determined as $\tpl{e}$ might not fully belong to the domain of $\beta$ (in such a case, each $u_i$ is internal).
    
    Let $\gamma: V(\str{G}_1) \to V(\str{G}_2)$ be the partial isomorphism from the game $H_\str{G}$, which maps the picked vertices and also corresponding roots to each other.
    As the $S$-edge on $\tpl{u}$ came from $\str{G}_1$, we have $\iota_\tpl{e}(\tpl{u}) \in S^{\str{G}_1}$.
    Since $\gamma$ is a partial isomorphism and $\iota_\tpl{e}(\tpl{u})$ belongs to its domain, it follows that $\gamma(\iota_\tpl{e}(\tpl{u})) \in S^{\str{G}_2}$.
    Finally, observe that $\gamma(\iota_\tpl{e}(\tpl{u})) = \iota_\tpl{f}(\tpl{v})$ by Algorithm~\ref{alg:duplicator_strategy_simple}.
    Hence, $\tpl{v} \in S^{\str{A}_2*\str{G}_2}$, which concludes the proof.
\end{proof}

We remark that a similar statement (with different bounds) can be proved via interpretations, which is another important model-theoretic tool that allows to transfer properties from one structure to another by \emph{defining} the latter structure in the former \cite{hodges}\cite{unified_approach}.
The construction of an appropriate interpretation of $\str{A} * \str{G}$ in the disjoint union of $\str{A}$ and $\str{G}$ follows the set-wise definition of $\str{A} * \str{G}$.
Such an interpretation also implies some weak results about preservation of $\FO$-convergence.

Nevertheless, we consider the $\EF$-games to be a more suitable tool for our purposes.
It provides us with fine-grained control which makes possible to prove more.
For instance, the results of the following section seem to be out of reach for interpretations.

\subsection[Fragmentation of R-edges]{Fragmentation of $R$-edges}\label{ssec:fragmentation}

Here we define \emph{fragmentation} of $R$-edges with the aim to give a sufficient condition for the elementary convergence of the sequence $\strseq{A} * \strseq{G}$ without requiring the elementary convergence of $\strseq{A}$.
We show that it is possible to remove a certain kind of information from the structures of $\strseq{A}$, which is irrelevant for the limit behavior of the sequence $\strseq{A} * \strseq{G}$.
The excessive information is the precise arrangement of $R$-edges, which we discard by their fragmentation.
The relaxed assumption then states that the elementary convergence of the sequence of fragmented base structures is sufficient.

\subsubsection{Motivation}\label{sssec:motivation_for_fragmentation}

The first-order logic is inherently local.
The Gaifman theorem states that any sentence can be expressed as a boolean combination of sentences of the form
\begin{equation}\label{eq:gaifman_sentence}
    \exists \tpl{y} 
    \left(
        \bigland_{1 \leq i < j \leq |\tpl{y}|} \dist(y_i, y_j) > 2r
        \land
        \bigland_{1 \leq i \leq |\tpl{y}|} \psi(y_i)
    \right)
    ,
\end{equation}
where the formula $\psi$ is $r$-local.

Consider the following example, where all the graphs are simple and undirected.

\begin{example}
    Let $\strseq{A}$ be a sequence of $d$-regular graphs with an increasing number of vertices.
    Let $\strseq{G}$ be a sequence of paths of increasing length with the endpoints as the roots.
    We claim that the sequence $\strseq{A}*\strseq{G}$ is elementarily convergent.

    For a sufficiently large $n \in \N$, we can distinguish vertices of $\str{A}_n * \str{G}_n$ by their $r$-neighborhood into those in distance $\ell \leq r$ from an internal vertex and the others whose $r$-neighborhood is a path.
    Let $\phi$ be a sentence of the form~\eqref{eq:gaifman_sentence}.
    Either the $r$-local formula $\psi$ is satisfied on vertices of one of these kinds, then $\str{A}_n * \str{G}_n \models \phi$ (as there is enough vertices of each kind), or $\str{A}_n * \str{G}_n \not\models \phi$.
    This is true for any large enough $n$; therefore, $\stonepar{\phi}{\strseq{A}*\strseq{G}}$ converges and $\strseq{A}*\strseq{G}$ is elementarily convergent.
\end{example}

Notice that the example contains an assumption only on the degrees of internal vertices while the exact interconnection of $R$-edges in $\str{A}_n$ is irrelevant.
Each individual internal vertex in $\str{A}_n*\str{G}_n$ sees how many gadget copies are attached to it.
However, as the gadgets grow and their roots tend away from each other, it becomes impossible the tell where the other ends of the gadget copies are attached.
This phenomenon is apparent in the limit: the distance of the gadget's roots grows to $+\infty$, which implies that they lie in distinct connected components of $\ellim \strseq{G}$.
The elementary limit of $\strseq{A}*\strseq{G}$ is an infinite collection of stars, each with $d$ infinite rays, with no connection among them.

This effect of growing gadgets occurs also in the general setting.
Let $\strseq{G}$ be an elementarily convergent sequence of gadgets.
We define $\sigma$ to be the equivalence on $[\arity{R}]$ with
\begin{equation}\label{eq:canonical_sigma}
\begin{aligned}
    (i,j) \in \sigma 
    &\Leftrightarrow 
    \lim \dist_\strseq{G}(z_i^\strseq{G}, z_j^\strseq{G}) < \infty \\
    &\Leftrightarrow
    z_i, z_j \text{ share a connected component in } \ellim \strseq{G}
    .
\end{aligned}
\end{equation}
We denote this canonical equivalence for the sequence $\strseq{G}$ by $\Eq(\strseq{G})$.
Abusing notation slightly, if the indices $i$ and $j$ are $\sigma$-equivalent, we also say that the roots $z_i$ and $z_j$ are $\sigma$-equivalent.

It is clear, at least if $\strseq{A}$ is elementarily convergent, that the exact positions of $R$-edges in $\ellim \strseq{A}$, which we denote by $\str{B}$, are irrelevant.
Only the the positions of \emph{subedges} that gather the vertices of $R$-edges on $\sigma$-equivalent indices matter.
In particular, suppose we permute the interconnection of $R$-edges in $\str{B}$, obtaining a structure $\str{C}$, in such a way that we preserve the subedges.
That is, for each class $X \subseteq [\arity{R}]$ of $\sigma$ there is a bijection $f_X: R^\str{B} \to R^\str{C}$ satisfying that
\[
    \forall \tpl{e} \in R^\str{B}
    \;
    \forall i \in X : \tpl{e}(i) = f_X(\tpl{e})(i)
    ,
\]
where we use the function notation for the tuples.
Then the structure $\str{C} * (\ellim \strseq{G})$ is exactly the same as $\str{B} * (\ellim \strseq{G})$.

Our goal is to draw this observation to the finite case, when the distances between $\sigma$-nonequivalent roots are large but possibly finite, and prove that the resulting structures are difficult to distinguish.
We start by defining the structure $\str{A}^\sigma$ which preserves the full information about the subedges from the base structure $\str{A}$ with respect to the equivalence $\sigma$.
We prove that the elementary convergence of the sequence of fragmented structures $\strseq{A}^\sigma$, together with the elementary convergence of $\strseq{G}$, is sufficient for the elementary convergence of $\strseq{A}*\strseq{G}$.

\subsubsection{Fragmentation}\label{sssec:fragmentation}

Let us define the fragmented base structure $\str{A}^\sigma$.
Note that we cannot simply project the $R$-edges to the $\sigma$-equivalent indices as that would lose track of their multiplicities.
Instead, we add an \emph{auxiliary} vertex to each subedge, which allows us to discern individual subedges.

We start by the definition of the language of $\str{A}^\sigma$.

\begin{definition}[Language of fragmented base structures]
    Let $\sigma$ be an equivalence on $[\arity{R}]$ with classes $X_1, X_2, \dots, X_\ell$.
    Additionally, we consider $X_0 = \emptyset$ to be a class of $\sigma$.
    Let $L_\sigma$ be the extension of $L$ by symbols $R_i$ of arity $|X_i| + 1$ for $i \in [\ell]_0$.
\end{definition}

The $+1$ in the arity of $R_i$ is for the auxiliary vertex.
The structure $\str{A}^\sigma$ can be formally defined as the result of gadget construction applied to $\str{A}$ with a certain canonical gadget for the equivalence $\sigma$.

\begin{definition}[Fragmentation]\label{def:fragmentation}
    Denote by $\Gad(\sigma)$ the gadget with
    \[
        V(\Gad(\sigma)) = \{z_1, \dots, z_{\arity{R}}\} \cup \{x_0, x_1, \dots, x_\ell\},
    \]
    where the vertices $z_j$ are the roots.
    There is exactly one $R_i$-edge for each $i \in [\ell]_0$ spanning the vertices $z_j, j \in X_i,$ and the vertex $x_i$.
    
    Let $\str{A}$ be a base structure.
    We write $\str{A}^\sigma$ for the $L_\sigma$-structure $\str{A} * \Gad(\sigma)$.
    A structure of the form $\str{A}^\sigma$ is called a \emph{fragmented} base structure.
\end{definition}

\missingfigure{Example of fragmentation, e.g. for the $\overrightarrow{C}_3$ example.}

As indicated, we call the $R_i$-edges from a copy of $\Gad(\sigma)$ replacing an $R$-edge $\tpl{e}$ the \emph{subedges} of $\tpl{e}$.
The $R_i$-subedge of $\tpl{e}$ for a class $X_i$ is denoted by $\tpl{e}_{X_i}$.
Conversely, $\tpl{e}$ is the superedge of $\tpl{e}_{X_i}$.
The vertices in $V(\str{A}^\sigma) \setminus V(\str{A})$, i.e. the copies of $x_0, \dots, x_\ell$, are called \emph{auxiliary}.

\begin{remark}
    The sole purpose of the auxiliary vertices is to record the number of subedges.
    An equivalent approach to the definition would be to allow multiedges by using many-sorted logic, where the vertices and edges are considered to be distinct entities in the universe of a structure.
    Then the subedges would be truly defined as a projection of $R$-edges.
    The auxiliary vertices allow us to stay in the usual one-sorted logic, although they admittedly bring their own technical challenges.
\end{remark}

The following theorem, shows that only the information is the structures $\str{A}^\sigma$ is necessary for the behavior of $\str{A}*\str{G}$ provided that the $\sigma$-nonequivalent roots are far apart.

Recall that $\maxarity{\lambda}$ stands for the maximum arity of a symbol from $\lambda$.

\begin{theorem}\label{thm:elementary_equivalence_using_fragmentation}
    Fix $k \in \N$.
    Let $\str{A}_1, \str{A}_2$ be base structures, $\str{G}_1, \str{G}_2$ gadgets and $\sigma$ an equivalence on $[\arity{R}]$ whose maximal class has size $m$.
    Suppose it holds
    \begin{equation*}
        \begin{gathered} 
            \str{A}^\sigma_1 \equiv_{(m+1)k} \str{A}^\sigma_2
            , \\
            \str{G}_1 \equiv_{2^{k+1}\cdot\maxarity{L_G}} \str{G}_2
            , \\
            \forall\; i,j \in [\arity{R}]: \dist_{\str{G}_1}(z_i^{\str{G}_1}, z_j^{\str{G}_1}) \leq 2^{k+1} \Rightarrow (i,j) \in \sigma
            .
        \end{gathered}
    \end{equation*}
    Then we have
    \[
        \str{A}_1 * \str{G}_1 \equiv_k \str{A}_2 * \str{G}_2
        .
    \]
\end{theorem}

We leave the proof of Theorem~\ref{thm:elementary_equivalence_using_fragmentation}, which starts by showing that the second assumption implies the third for $\str{G}_2$, for Section~\ref{sssec:proof_of_fragmentation}.
Now we proceed to the statement about the elementary convergence and limit of the sequence $\strseq{A} * \strseq{G}$.
Already, Theorem~\ref{thm:elementary_equivalence_using_fragmentation} implies that elementarily convergence of sequences $\strseq{A}^\sigma$ and $\strseq{G}$ ensure elementarily convergent sequence $\strseq{A}*\strseq{G}$ (provided that $\sigma = \Eq(\strseq{G})$).
It is rather intuitive that the limit of $\strseq{A}*\strseq{G}$ should be obtained by applying gadget construction to the elementary limits of $\strseq{A}^\sigma$ and $\strseq{G}$.
Strictly speaking, this is not a classical gadget construction as the structure $\ellim \strseq{G}$ is a gadget designed to replace $R$-edges while the structure $\ellim \strseq{A}^\sigma$ is only a fragmented base structure (in particular, it does not contain $R$-edges).
Nevertheless, the intended result is clear: replace each $R_i$-edge by the component of $\ellim \strseq{G}$ that contains the roots from $X_i$ and remove the auxiliary vertices in the process.
In particular, the $R_0$-edges are replaced by the union of components of $\ellim \strseq{G}$ that contain no root.
We denote this modified gadget construction by $*_\sigma$.

\begin{theorem}\label{thm:preservation_of_elementary_convergence_with_fragmentation}
    Let $\strseq{A}$ be a sequence of base structures and $\strseq{G}$ be an elementarily convergent sequence of gadgets.
    Set $\sigma = \Eq(\strseq{G})$.
    If $\strseq{A}^\sigma$ is elementarily convergent, then the sequence $\strseq{A} * \strseq{G}$ is elementarily convergent and we have
    \[
        \ellim (\strseq{A}*\strseq{G}) \equiv (\ellim \strseq{A}^\sigma) *_\sigma (\ellim \strseq{G})
        .
    \]
\end{theorem}
\begin{proof}
    As noted above, a direct application of Theorem~\ref{thm:elementary_equivalence_using_fragmentation} yields that $\strseq{A}*\strseq{G}$ is elementarily convergent.
    In the rest of the proof, we show that the elementary limit can be expressed as $(\ellim \strseq{A}^\sigma) *_\sigma (\ellim \strseq{G})$.
    In particular, we show that $(\ellim \strseq{A}^\sigma) *_\sigma (\ellim \strseq{G})$ is the elementary limit of $\strseq{A}_{\bm{f}}*\strseq{G}_{\bm{f}}$, where $\strseq{A}_{\bm{f}} = (\str{A}_{f(n)})_{n \in \N}$ is an elementarily convergent subsequence of $\strseq{A}$.
    This is sufficient as the sequence $\strseq{A}*\strseq{G}$ is elementarily convergent.

    So, let $\strseq{A}_{\bm{f}}$ be an elementarily convergent subsequence of $\strseq{A}$.
    There is one due to the compactness of the space $(\mathcal{B}, \rho_\mathcal{B})$.
    Corollary~\ref{cor:preservation_of_elementary_convergence} states that
    \begin{equation}\label{eq:commutation_of_gadget_construction_and_limit}
        (\ellim \strseq{A}_{\bm{f}} * \strseq{G}_{\bm{f}}) = (\ellim \strseq{A}_{\bm{f}}) * (\ellim \strseq{G}_{\bm{f}})
        .
    \end{equation}
    The operation $*_\sigma$ is defined in such a way that it holds
    \begin{equation}\label{eq:exchange_of_fragmentation_and_gadget_construction}
              (\ellim \strseq{A}_{\bm{f}}) * (\ellim \strseq{G}_{\bm{f}}) 
        \cong (\ellim \strseq{A}_{\bm{f}})^\sigma *_\sigma (\ellim \strseq{G}_{\bm{f}})
    \end{equation}
    provided we use isomorphic limit structures on both sides (and not just elementarily equivalent).
    Moreover, we have
    \begin{equation}\label{eq:equivalence_from_subsequence_to_sequence}
    \begin{aligned}
                (\ellim \strseq{A}_{\bm{f}})^\sigma *_\sigma (\ellim \strseq{G}_{\bm{f}})
        &\equiv (\ellim \strseq{A}_{\bm{f}}^\sigma) *_\sigma (\ellim \strseq{G}_{\bm{f}}) \\
        &\equiv (\ellim \strseq{A}^\sigma)          *_\sigma (\ellim \strseq{G})
    \end{aligned}
    \end{equation}
    In both equalities, we are interchanging elementarily equivalent structures, which is possible by Theorem~\ref{thm:continuity_of_gadget_construction} as $*_\sigma$ is essentially a repeated use of $*$ (the additional removal of auxiliary vertices from structures on both sides does not harm the elementary equivalence).
    In particular, for the first equality, observe that $(\ellim \strseq{A}_{\bm{f}})^\sigma \equiv \ellim \strseq{A}_{\bm{f}}^\sigma$ due to the definition of fragmented base structures (via gadget construction) and Corollary~\ref{cor:preservation_of_elementary_convergence}.
    In the second one, we utilize that $\strseq{A}_{\bm{f}}^\sigma$ and $\strseq{G}_{\bm{f}}$ are subsequences of convergent sequences $\strseq{A}^\sigma$ and $\strseq{G}$.
    Combining~\eqref{eq:commutation_of_gadget_construction_and_limit},\eqref{eq:exchange_of_fragmentation_and_gadget_construction} and~\eqref{eq:equivalence_from_subsequence_to_sequence}, we reach the conclusion.
\end{proof}

We remark that a similar statement also holds for the ultraproducts.

\subsubsection{Proof of Theorem~\ref{thm:elementary_equivalence_using_fragmentation}}\label{sssec:proof_of_fragmentation}

To a large degree, we follow the lines of the proof of Lemma~\ref{lem:continuity_bound_via_ef_games}.
The main difference is due to fact that the $\EF$-game on $\str{A}_1^\sigma$ and $\str{A}_2^\sigma$ allows to identify only corresponding pairs of subedges, but not of the whole $R$-edges.
We introduce a new mechanism that assigns to a picked external vertex from $\str{G}_S^\tpl{e}$ a subedge of $\tpl{e}$.
Finding the corresponding subedge identifies a copy $\str{G}_D^\tpl{f}$ where we look for the Duplicator's answer.
Beware that subedges of a single $R$-edge $\tpl{e}$ may correspond to subedges of several distinct $R$-edges; we need to ensure that Spoiler is not able to exploit such a discrepancy.

We start by simple lemmas about distances.

\begin{lemma}\label{lem:measuring_distances}
    Let $\str{A}$ and $\str{B}$ be $\lambda$-structures containing vertices $a_1, a_2$, resp. $b_1, b_2$.
    Suppose that Duplicator wins $\EF_k(\str{A}, a_1, a_2; \str{B}, b_1, b_2)$.
    For $r \in \N$ satisfying that $r \cdot \maxarity{\lambda} \leq k$, we have either
    \[
        \dist_\str{A}(a_1, a_2) = \dist_\str{B}(b_1, b_2)
        ,
    \]
    or
    \[
        \dist_\str{A}(a_1, a_2) > r \text{ and } \dist_\str{B}(b_1, b_2) > r
        .
    \]
\end{lemma}
\begin{proof}
    Suppose that $d = \dist_\str{A}(a_1, a_2) \leq r$.
    Then there is a path $\tpl{e}_1, \dots, \tpl{e}_d$ connecting $a_1, a_2$ in $\str{A}$ (i.e. $\tpl{e}_i$ is an $S_i$-edge for some $S_i \in \lambda$, $a_1 \in \tpl{e}_1, a_2 \in \tpl{e}_d$ and each $\tpl{e}_i, \tpl{e}_{i+1}$ share at least one vertex).
    Spoiler have enough rounds to pick all the vertices of edges $\tpl{e}_1, \dots, \tpl{e}_d$.
    Since Duplicator has a winning strategy, there is a path $\tpl{f}_1, \dots, \tpl{f}_d$ in $\str{B}$ connecting $b_1, b_2$.
    Therefore, $\dist_\str{A}(b_1, b_2) \leq d$ and the converse inequality follows by the symmetric argument.
\end{proof}

When the assumptions of the lemma arise, we say that we can \emph{measure distances up to $r$} in the given game.

Fix $r \in \N$ and let $\str{A}$ be a $\lambda$-structure with $M \subseteq V(\str{A})$ and a coloring $c: M \to [n]$.
If for each $u,v \in M$ with $c(u) \not= c(v)$ holds that $\dist_\str{A}(u,v) > r$, we say that the coloring $c$ is \emph{$r$-discrete}.

\begin{lemma}\label{lem:about_discrete_colorings}
    Let $\str{A}$ be a $\lambda$-structure with a $2r$-discrete coloring $c$ on $M \subseteq V(\str{A})$.
    Suppose we color a vertex $v \in V(\str{A}) \setminus M$ by the following rule:
    if there is $u \in M$ with $\dist_\str{A}(v,u) \leq r$, set $c(v) = c(u)$.
    Otherwise, $v$ gets an arbitrary color.
    Then, the resulting coloring on $M \cup \{v\}$ is $r$-discrete.
\end{lemma}
\begin{proof}
    Directly follows from the triangle inequality for $\dist_\str{A}(\cdot, \cdot)$.
\end{proof}

Now we are ready to give the main proof of this section.

\begin{proof}[Proof of Theorem~\ref{thm:elementary_equivalence_using_fragmentation}]
    Set $H = \EF_k(\str{A}_1*\str{G}_1; \str{A}_2*\str{G}_2)$, $H_{\str{A}^\sigma} = \EF_{(m+1)k}(\str{A}^\sigma_1;\str{A}^\sigma_2)$ and $H_\str{G} = \EF_{2^{k+1}\cdot\maxarity{L_G}}(\str{G}_1;\str{G}_2)$.
    We use Lemma~\ref{lem:about_discrete_colorings} to color the picked vertices in $\str{G}_1$ and $\str{G}_2$ by equivalence classes of $\sigma$; in particular, if a vertex is allowed to get an arbitrary color, we use the color $X_0$.
    The lemma is applied independently for vertices from $\str{G}_1$ and $\str{G}_2$, however, we will prove that the colors assigned to both vertices picked in a single round are the same.  
    Initially, we assign to each root $z_i$ the color $X_j$ for which $i \in X_j$.
    Note that this initial coloring $c_1$ of $\str{G}_1$ is $2^k$-discrete due to the last assumption of the theorem.
    The second assumption together with Lemma~\ref{lem:measuring_distances} implies the same for the coloring $c_2$ of $\str{G}_2$.
    We argue that Algorithm~\ref{alg:duplicator_strategy_fragmentation} poses a winning strategy for Duplicator.
    
\begin{algorithm}
    \caption{Duplicator's response in one round of $H$}
    \label{alg:duplicator_strategy_fragmentation}
    \begin{algorithmic}[1]
        \State $S, D \gets$ indices of the Spoiler's and Duplicator's structure in this round
        \State $u \gets$ vertex chosen by Spoiler from $\str{A}_S * \str{G}_S$ 
        \If{$u$ is internal}
            \State Let Spoiler pick $\str{A}_S^\sigma$ and the vertex $u$ in $H_{\str{A}^\sigma}$
            \State $v \gets$ Duplicator's response in $H_{\str{A}^\sigma}$
        \Else
            \State $\tpl{e} \gets \rho(u)$
            \State Let Spoiler pick $\str{G}_S$ and the vertex $\iota_\tpl{e}(u)$ in $H_\str{G}$
                \label{step:spoiler_selects_in_HG}
            \State $v' \gets$ Duplicator's response in $H_\str{G}$
                \label{step:duplicator_selects_in_HG}
            \State $X \gets c_D(v')$ (the color assigned to $v'$ in $\str{G}_D$)
            \State Let Spoiler pick $\str{A}_S^\sigma$ and all the vertices of $\tpl{e}_X$ in $H_{\str{A}^\sigma}$
            \State $\tpl{f}' \gets$ Duplicator's response in $H_{\str{A}^\sigma}$
            \State $\tpl{f} \gets$ the superedge of $\tpl{f}'$ from $R^{\str{A}_D}$
            \State $v \gets \iota_{\tpl{f}}^{-1}(v')$
        \EndIf
        \State Vertex $v$ is the Duplicator's response
    \end{algorithmic}
\end{algorithm}

Most of the reasoning the same as in the proof of Lemma~\ref{lem:continuity_bound_via_ef_games}.
First of all, we exceed the length of neither $H_{\str{A}^\sigma}$ nor $H_\str{G}$.
Let $\tpl{a}$ and $\tpl{b}$ be the $t$-tuples chosen from $\str{A}_1*\str{G}_1$ and $\str{A}_2*\str{G}_2$ after $t$ rounds and $\tpl{u}$ an $S$-edge in $\str{A}_1*\str{G}_1$ with $u_j = a_{i_j}$ for all $j \in [s]$ for some indices $i_1, \dots, i_s \in [t]$ (where $s = \arity{S}$).
We prove that $\tpl{v}$ with $v_j = b_{i_j}$ is an $S$-edge in $\str{A}_2*\str{G}_2$.
If $\tpl{u}$ originated from $\str{A}_1$, the same argument as in Lemma~\ref{lem:continuity_bound_via_ef_games} applies.

We consider the case when $\tpl{u}$ originated from $\str{G}_1$, which needs to be handled more carefully.
Suppose that the $S$-edge arrived within a copy $\str{G}_1^\tpl{e}$.
First, we observe that the $S$-edge $\iota_\tpl{e}(\tpl{u})$ in $\str{G}_1$ is monochromatic, i.e. all vertices were assigned the same color in the game $H_\str{G}$.
This follows from Lemma~\ref{lem:about_discrete_colorings}: the colors were initially $2^k$-discrete and $2^{k-t}$-discrete after $t \leq k$ rounds; thus, they are at least $1$-discrete, which implies that vertices of distinct colors cannot share an edge.
Moreover, we claim that the color assigned to the vertices from Steps~\ref{step:spoiler_selects_in_HG}~and~\ref{step:duplicator_selects_in_HG} is the same.
This is proved by induction using Lemma~\ref{lem:measuring_distances}.
Initially, the colors of corresponding roots are the same.
In the $i$-th round of the game $H_\str{G}$, we can measure distances at least up to $2^{k-i}$.
Hence, if the picked vertex, say $w_i$ from $\str{G}_S$ gets color $X$ as being close (in distance at most $2^{k-i}$) to a vertex $w_j$ picked in $j$-th round, then Duplicator is obliged, by Lemma~\ref{lem:measuring_distances}, to pick a vertex $w'_i$ with $\dist_{\str{G}_1}(w_i, w_j) = \dist_{\str{G}_2}(w'_i, w'_j)$, where $w'_j$ is the vertex picked in the $j$-th round from $\str{G}_D$.
As a result, $c_S(w_i) = c_D(w'_i)$.
If $w_i$ was far from all colored vertices, then so does $w'_i$; hence, $c_S(w_i) = c_D(w'_i) = X_0$.

Let $\beta: V(\str{A}_1^\sigma) \to V(\str{A}_2^\sigma)$ and $\gamma: V(\str{G}_1) \to V(\str{G}_2)$ be the partial isomorphisms from games $H_{\str{A}^\sigma}$ and $H_\str{G}$.
Using the observation above, we deduce that all the vertices $\gamma(\iota_\tpl{e}(\tpl{u}))$ have the same color as the vertices $\iota_\tpl{e}(\tpl{u})$.
It follows that $\gamma(\iota_\tpl{e}(\tpl{u})) = \iota_\tpl{f}(\tpl{v})$, where $\tpl{f} \in R^{\str{A}_2}$ satisfies $\beta(\tpl{e}_X) = \tpl{f}_X$.
Therefore, the tuple $\iota_\tpl{f}(\tpl{v})$ and consequently the tuple $\tpl{v}$ form an $S$-edge in $\str{G}_2$ and $\str{A}_2 * \str{G}_2$, respectively, which concludes the proof.
\end{proof}
\section{Obstacles to local convergence}\label{sec:obstacles}

Here we demonstrate that local convergence needs not to be preserved by gadget construction and show some general reasons why: fluctuating proportion of internal and external vertices and magnification of zero-measure differences for $R$-edges.
Moreover, we give a particular example where local convergence of a sequence of graphs is broken by subdividing each edge by one vertex, which is a simple case of gadget construction with a constant gadget.
We view this section as useful preparation for the following one, where we discuss sufficient conditions for obtaining local convergence.

The structures constructed in the examples below are undirected, see remark in Section~\ref{ssec:gadget_construction}.
Note that although we focus on local convergence, all the sequences $\strseq{A}$ and $\strseq{G}$ bellow are also elementarily convergent.

\subsection{Fluctuating proportion of internal and external vertices}\label{ssec:fluctuating_proportion_of_internal_and_external_vertices}

One obstacle for local convergence is the fluctuating proportion of internal and external vertices in the sequence $\strseq{A}*\strseq{G}$.
In general, the patterns that appear in $\strseq{A}$ and $\strseq{G}$ may differ.
Thus, if the proportion of internal and external vertices fluctuates, it is likely that the sequence $\strseq{A}*\strseq{G}$ is not local convergent as the probability of observing a certain pattern varies.
Such examples with fluctuating proportion are easy to construct: we consider sequences $\strseq{A}$ and $\strseq{G}$ with $|V(\str{A}_n)| \ll |V(\str{G}_n)|$ for odd $n$, and $|V(\str{A}_n)| \gg |V(\str{G}_n)|$ for even $n$.

\begin{example}\label{ex:basic_fluctuation}
    Let $R$ be a unary symbol.
    Consider the following sequence of base graphs:
    \begin{align*}
        \str{A}_n =
        \begin{cases}
            K_n\text{ with an arbitrary vertex marked by $R$} & \text{if $n$ is odd}
            , \\
            K_{2^n}\text{ with an arbitrary vertex marked by $R$} & \text{if $n$ is even}
            .
        \end{cases}
    \end{align*}
    The sequence of gadgets is defined similarly.
    Let $S_n$ be the star with $n$ leaves.
    \begin{align*}
        \str{G}_n =
        \begin{cases}
            S_{2^n}\text{ with the inner vertex as the root} & \text{if $n$ is odd}
            , \\
            S_n\text{ with the inner vertex as the root} & \text{if $n$ is even}
            .
        \end{cases}
    \end{align*}
    Both sequences are local convergent as asymptotically almost all $p$-tuples are the same, i.e. exchangeable by an automorphism.
    However, the sequence $\strseq{A}*\strseq{G}$ is not $\FOloc{1}$-convergent, which is witnessed by the formula $\phi(x)$ stating ``the degree of $x$ is $1$''.
\end{example}

This obstacle may also cause the fail of local convergence in a more subtle context.
In the following example, we consider the operation of $1$-subdivision of edges.
Note that it is a special case of gadget construction with the gadget formed by a path of length $2$ with the endpoints as the roots.

The $(k,\ell)$-lollipop graph $L_{k,\ell}$ is the graph composed of a clique on $k$ vertices and a path of length $\ell$ that share a single vertex, an endpoint of the path.

\begin{example}\label{ex:counterexample_subdivision}
    We define $\strseq{A}$ as the following sequence of lollipop graphs:
    \begin{align*}
        \str{A}_n =
        \begin{cases}
            L_{n, n^3} & \text{if $n$ is odd}, \\
            L_{n, n^{3/2}} & \text{if $n$ is even}.
        \end{cases}
    \end{align*}
    The sequence $\strseq{A}$ is local convergent:
    the $r$-neighborhood of $p$ uniformly chosen vertices is asymptotically almost surely a disjoint collection of paths.
    
    We claim that $\FOloc{1}$-convergence fails for the sequence $\strseq{A}^\bullet$ of $1$-subdivisions of $\strseq{A}$:
    for odd $n$ the path still dominates in the graphs $\str{A}^\bullet_n$ while for even $n$ dominates the subdivided clique.
    In particular, there is $\Theta(n^2)$ external vertices within the clique and only $\Theta(n^{3/2})$ of all the other vertices.
    Thus, we use the formula $\phi(x)$ stating ``$x$ has exactly two neighbors of degree $2$'' as a witness that the sequence $\strseq{A}^\bullet$ is not $\FOloc{1}$-convergent.
\end{example}

\subsection{Magnification of zero-measure differences}\label{ssec:magnification_of_zero_measure_differences}

A more intriguing obstacle is the magnification of zero measure differences of $R$-edges.
Suppose that the $R$-edges are sparse in the sequence $\strseq{A}$, i.e. $\lim \stonepar{R}{\str{A}_n} = 0$, where the symbol $R$ is considered as an atomic formula.
Even if $\strseq{A}$ is $\FO$-convergent, it is possible that the behavior of $R$-edges is far from stable.
That is, the probabilities
\[
    \Pr[\strseq{A} \models \phi(\tpl{x}) \;|\; \strseq{A} \models R(\tpl{x})]
\]
need not to converge (note that the condition has probability $0$).
However, when applying gadget construction, such discrepancies may be magnified and become of a non-zero measure.

\begin{example}\label{ex:magnification}
    Let $R$ be a unary symbol and suppose that $L$ contains a unary symbol $S$.
    We denote by $I_n$ the independent set on $n$ vertices and by $I_n^R$, resp. $I_n^{R,S}$, we indicate that the vertices of $I_n$ are marked by $R$, resp. by both $R$ and $S$.
    Consider the following sequence of base graphs:
    \begin{align*}
        \str{A}_n =
        \begin{cases}
            K_{n^2} \oplus I_n^R \oplus I_{2n}^{R,S}
            & \text{if $n$ is odd}
            , \\
            K_{n^2} \oplus I_{2n}^R \oplus I_n^{R,S}
            & \text{if $n$ is even}
            ,
        \end{cases}
    \end{align*}
    where $\oplus$ stands for the disjoint union.
    The gadget $\str{G}_n$ is the star $S_{2^n}$ with the inner vertex as the root.
    The sequence $\strseq{A}$ is local convergent by a similar argument as above.
    The external vertices dominate in the sequence $\strseq{A}*\strseq{G}$, which is again not $\FOloc{1}$-convergent.
    As the witness, we use the formula $\phi(x)$ stating ``$x$ has a neighbor marked by $S$''.
    
    Note that the same example works with $I_1^R, I_2^{R,S}$ and $I_2^R, I_1^{R,S}$, but such a sequence is not elementarily convergent.
\end{example}

This obstacle does not occur when the $R$-edges are dense in $\strseq{A}$, because the probability of the condition is positive and the conditional probabilities converge.
Moreover, we avoid the obstacle if the sequence $\strseq{A}$ is elementarily convergent and the number of $R$-edges in $\strseq{A}$ is bounded, which follows from the result in \cite{rooting_algebraic_vertices}.

\section{Positive cases of local convergence}\label{sec:positive_cases}

In this section, we study sufficient conditions for local convergence of the sequence $\strseq{A}*\strseq{G}$.
We start by showing that it is enough to avoid the obstacles from the previous section to obtain the convergence.
Then we give another sufficient condition that exploits the locality of the first-order logic.
These approaches are combined in the last part, using the idea of fragmentation from Section~\ref{ssec:fragmentation}.

\subsection{Avoiding obstacles}\label{ssec:avoiding_obstacles}

Here we establish the local convergence of $\strseq{A}*\strseq{G}$ provided that the known obstacles do not occur.
In order to draw the convergent behavior from sequences $\strseq{A}$ and $\strseq{G}$ to $\strseq{A}*\strseq{G}$, we define \emph{representation equivalence} that captures the local behavior of a $p$-tuple $\tpl{a}$ from $\str{A}_n * \str{G}_n$ using the representation of $\tpl{a}$ in the structures $\str{A}_n$ and $\str{G}_n$.
Given the absence of obstacles, the probability that a uniformly selected $p$-tuple belongs to a fixed class $\mathcal{C}$ of representation equivalence converges.
This, as we show, implies that the sequence $\strseq{A}*\strseq{G}$ is local convergent.

Let us define the notion of representation equivalence.
We actually consider a parameterized form: $(k,r,p)$-representation equivalence.
Loosely speaking, two $p$-tuples from $\str{A}_1*\str{G}_1$ and $\str{A}_2*\str{G}_2$ are $(k,r,p)$-representation equivalent if the $r$-neighborhoods of their representation in the structures $\str{A}_1, \str{G}_1$ and $\str{A}_2, \str{G}_2$ are $f(k,r,p)$-elementarily equivalent for some fixed function $f: \N^3 \to \N$.

The definition proceeds in several steps.

\begin{definition}[Profile]
    Let $\tpl{a}$ be a $p$-tuple from $\str{A}*\str{G}$.
    The profile of $\tpl{a}$ is an ordered partition $(I, E_1, \dots, E_t)$ of $[p]$ such that
    \begin{align*}
        I                   &= \{i : a_i\text{ is internal}\}
        , \\
        \bigcup_{j=1}^t E_j &= \{i : a_i\text{ is external}\}
        .
    \end{align*}
    Two indices $i,i'$ of external vertices $a_i, a_{i'}$ share a set $E_j$ if and only if $\rho(a_i) = \rho(a_{i'})$.
    The set $I$ is possibly empty while we require each $E_j$ being non-empty.
    The sets $E_1, \dots, E_t$ are listed by the ascending order of their minimal elements.
\end{definition}

We recall that an internal vertex $a$ from $\str{A}*\str{G}$ is represented in $\str{A}$ by itself.
An external vertex $a$ is represented by the $R$-edge $\tpl{e} = \rho(a)$ from $\str{A}$ and the vertex $\iota_{\tpl{e}}(a)$ from $\str{G}$.

Also recall that $(\str{A},a)$ denotes the structure $\str{A}$ rooted at $a$ and $\str{A}^r$ stands for the substructure of $\str{A}$ induced by the $r$-neighborhood of roots of $\str{A}$.

\begin{definition}[Representation]
    Let $\tpl{a}$ be a $p$-tuple from $\str{A}*\str{G}$ with the profile $(I,E_1, \dots, E_t)$.
    We define $\str{A}(\tpl{a}, r)$ to be the structure $(\str{A}, \tpl{b}_1, \dots, \tpl{b}_p)^r$, where
    \begin{align*}
        \tpl{b}_i = 
        \begin{cases}
            a_i                         & \text{if $a_i$ is internal}
            , \\
            \rho(a_i) & \text{if $a_i$ is external}
            .
        \end{cases}
    \end{align*}
    Moreover, we define for each $j \in [t]$ the structure $\str{G}^j(\tpl{a}, r)$ to be $(\str{G}, \tpl{c}_j)^r$, where $\tpl{c}_j$ is the tuple of vertices $\iota_{\rho(a_i)}(a_i)$ with $i \in E_j$.
\end{definition}

\begin{definition}[Representation equivalence]
    Let $\tpl{a}_1$ and $\tpl{a}_2$ be $p$-tuples from $\str{A}_1*\str{G}_1$ and $\str{A}_2*\str{G}_2$ with the same profile.
    We say that $\tpl{a}_1$ and $\tpl{a}_2$ are \emph{$(k,r,p)$-representation equivalent} if the following conditions hold:
    \begin{align*}
        \str{A}_1(\tpl{a}_1, r) &\equiv_{f(k,r,p)} \str{A}_2(\tpl{a}_2, r)
        , \\
        \forall j \in [t] : \str{G}_1^j(\tpl{a}_1, r) &\equiv_{f(k,r,p)} \str{G}_2^j(\tpl{a}_2, r)
        , \\
        \str{G}_1^r &\equiv_{f(k,r,p)} \str{G}_2^r
        ,
    \end{align*}
    where
    \[
        f(k,r,p) = (\arity{R})^{p+1}(r \cdot \maxarity{\lambda} + k)
        .
    \]
    In such a case, we write $(\str{A}_1*\str{G}_1, \tpl{a}_1) \approx_k^r (\str{A}_2*\str{G}_2, \tpl{a}_2)$.
\end{definition}

Observe that if $t \geq 1$, the last condition, which is necessary in general, follows from the previous one.
Also note that $\approx_k^r$ is an equivalence of finite index as it is based on $\equiv_{f(k,r,p)}$, which has finite index.

Let $\stonepar{\phi|_\pi}{\str{A}}$ denote the probability that a uniformly selected sequence (of tuples) $\tpl{b}_1, \dots, \tpl{b}_p$ from $\str{A}$ satisfies $\phi$, given that $(\str{A}, \tpl{b}_1, \dots, \tpl{b}_p)$ is a representation of some $p$-tuple $\tpl{a}$ from $\str{A}*\str{G}$ with the profile $\pi$.
That is, the probability
\[
    \Pr
    \left[
        \str{A} \models \phi(\tpl{b}_1, \dots, \tpl{b}_p) 
        \;\middle|\;
            i,i'\not\in I \Rightarrow             
            \big(
                \tpl{b}_i \in R^\str{A} 
                \text{ and } 
                \tpl{b}_i = \tpl{b}_{i'} \Leftrightarrow \exists j : i,i' \in E_j
            \big)
    \right]
    .
\]
We say that a profile $\pi$ is \emph{trivial} with respect to a sequence $\strseq{A}*\strseq{G}$ if the probability that a random $p$-tuple $\tpl{a}_n$ from $\str{A}_n*\str{G}_n$ has the profile $\pi$ tends to $0$.

The representation equivalence is key to obtain the following general theorem, whose proof we leave for Section~\ref{sssec:proof_of_general_local_convergence}.

\begin{theorem}\label{thm:general_local_convergence}
    Fix $p \in \N$.
    Let $\strseq{A}$ be a sequence of base structures and $\strseq{G}$ be a sequence of gadgets satisfying
    \begin{enumerate}[(i)]
        \item for every profile $\pi = (I,E_1, \dots, E_t)$ of a $p$-tuple that is non-trivial w.r.t. $\strseq{A}*\strseq{G}$ holds that for each $\phi \in \FOloc{|I|+(p-|I|)\arity{R}}$ the sequence $\stonepar{\phi|_\pi}{\strseq{A}}$ converges,
        \item $\strseq{G}$ is an $\FOcloc{m}$-convergent sequence of gadgets, where
        \[
            m = \max \{ |E_1| : \pi = (I,E_1, \dots, E_t) \text{ is non-trivial w.r.t. } \strseq{A}*\strseq{G} \}
            ,
        \]
        \item the proportion of internal vertices in $\strseq{A}*\strseq{G}$ tends to a limit.
    \end{enumerate}
    Then the sequence $\strseq{A}*\strseq{G}$ is $\FOloc{p}$-convergent.
\end{theorem}

We specialize the statement into several theorems with more natural assumptions.

\begin{theorem}\label{thm:fo_finitely_many_r_edges}
    Let $\strseq{A}$ be an $\FO$-convergent sequence of base structures and $\strseq{G}$ be an $\FOcloc{}$-convergent sequence of gadgets satisfying
    \begin{enumerate}[(i)]
         \item $\lim |R^\strseq{A}| = r < \infty$,
         \item the sequence $|V(\strseq{A})|\:/\:|V(\strseq{G})|$ has a limit.
    \end{enumerate}
    Then the sequence $\strseq{A}*\strseq{G}$ is local convergent.
\end{theorem}
\begin{proof}
    We apply Theorem~\ref{thm:general_local_convergence}; only the first assumption need to be verified.    
    Fix $p \in \N$, a profile $\pi = (I,E_1, \dots, E_t)$ with $t \leq r$, and a local formula $\phi(\tpl{x}_1, \dots, \tpl{x}_p) \in \FOloc{|I|+(p-|I|)\arity{R}}$.
    Without loss of generality, assume that $|I| = \{ p-|I| + 1, \dots, p \}$.
    We use the result from \cite{rooting_algebraic_vertices} to obtain an $\FO$-convergent sequence $\strseq{A}^+$ of lifts of $\strseq{A}$ with the property that each $R$-edge of $\str{A}_n$ is marked by a constant in $\str{A}_n^+$.
    (Note that although \cite{rooting_algebraic_vertices} assumes that the sequence $\strseq{A}$ has a limit structure, the limit statistics are sufficient for producing the lifts $\strseq{A}^+$.)
    
    It is possible to express the probability $\stonepar{\phi|_\pi}{\str{A}_n}$ as the sum over all choices of $R$-edges for the variables $\tpl{x}_1, \dots, \tpl{x}_{p-|I|}$ that form the representation of a $p$-tuple with the profile $\pi$.
    There is a finite number of such choices and the probability for each choice is computed by $\stonepar{\phi'}{\str{A}_n^+}$, where $\phi' \in \FOloc{|I|}$ is the formula $\phi$ after an appropriate substitution of constants for the variables $\tpl{x}_1, \dots, \tpl{x}_{p-|I|}$.
    Since $\strseq{A}^+$ is $\FO$-convergent, each of these sequences converge and their (finite) sum converges as well.
    Thus, the first assumption of Theorem~\ref{thm:general_local_convergence} is satisfied.
\end{proof}

Note that Example~\ref{ex:basic_fluctuation} (in the modified version with finitely many $R$-edges) shows that it is not possible to omit the assumption of elementary convergence of $\strseq{A}$.

In the following theorem, the dominance of internal vertices allows to reduce the assumption on $\strseq{G}$.

\begin{theorem}\label{thm:fo_dominant_internal_vertices}
    Let $\strseq{A}$ be a local convergent sequence of base structures and $\strseq{G}$ be an $\FOcloc{0}$-convergent sequence of gadgets such that the limit proportion of internal vertices in $\strseq{A}*\strseq{G}$ is $1$.
    Then the sequence $\strseq{A}*\strseq{G}$ is local convergent.
\end{theorem}
\begin{proof}
    Fix $p \in \N$.
    The only non-trivial profile w.r.t. $\strseq{A}*\strseq{G}$ is $\pi = ([p])$ thus the first assumption of Theorem~\ref{thm:general_local_convergence} reduces to $\FOloc{p}$-convergence of $\strseq{A}$ and the second to $\FOcloc{0}$-convergence of $\strseq{G}$.
\end{proof}

If the number of $R$-edges tends to infinity, the constant-local convergence of $\strseq{G}$ reduces to $\FOcloc{1}$-convergence.

\begin{theorem}\label{thm:fo_general_conditioning_with_infinitely_many_r_edges}
    Let $\strseq{A}$ be a local convergent sequence of base structures and $\strseq{G}$ be an $\FOcloc{1}$-convergent sequence of gadgets satisfying
    \begin{enumerate}[(i)]
        \item for every profile $\pi = (I,E_1, \dots, E_t)$ with all $|E_j| = 1$ holds that for each $\phi \in \FOloc{|I|+(p-|I|)\arity{R}}$ the sequence $\stonepar{\phi|_\pi}{\strseq{A}}$ converges,
        \item the proportion of internal vertices in $\strseq{A}*\strseq{G}$ tends to a limit,
        \item $\lim |R^\strseq{A}| = \infty$.
    \end{enumerate}
    Then the sequence $\strseq{A}*\strseq{G}$ is local convergent.
\end{theorem}
\begin{proof}
    This follows directly from Theorem~\ref{thm:general_local_convergence} as only the profiles with all $|E_j| = 1$ are non-trivial.
\end{proof}

A combination of these statements stems a pleasing corollary.

\begin{corollary}\label{cor:fo_dense_edges}
    Let $\strseq{A}$ be a local convergent sequence of base structures satisfying $\lim \stonepar{R}{\strseq{A}} > 0$ and $\strseq{G}$ be a constant-local convergent sequence of gadgets.
    Then the sequence $\strseq{A}*\strseq{G}$ is local convergent.
    
    Moreover, if $|V(\strseq{A})| \to \infty$, $\FOcloc{1}$-convergence of $\strseq{G}$ suffices for the conclusion.
\end{corollary}
\begin{proof}
    If the size of structures in $\strseq{A}$ is bounded, the sequence is eventually constant (which is implied even by $\FOloc{2}$-convergence).
    Thus, Theorem~\ref{thm:fo_finitely_many_r_edges} applies: either the gadgets grow or are eventually constant as well.
    In both cases, the second assumption of the theorem is satisfied.
    
    Otherwise, it holds $\lim |R^\strseq{A}| = \infty$ and we use Theorem~\ref{thm:fo_general_conditioning_with_infinitely_many_r_edges}.
    The first assumption is satisfied due to the fact that conditioning on the selection of an $R$-edge in $\strseq{A}$ is possible:
    the event that a random $\arity{R}$-tuple forms an $R$-edge has positive probability.
    It remains to verify the last assumption of the theorem.
    We distinguish several cases to deduce the limit proportion $c$ of internal vertices.
    If $\lim |V(\strseq{G})| = \arity{R}$, i.e. $\strseq{G}$ eventually contains only roots, then $c = 1$.
    Assume otherwise.
    If $\arity{R} > 1$, we have $c = 0$.
    If $\arity{R} = 1$, then either $\lim |V(\strseq{G})| = \infty$ and $c = 0$, or $\lim |V(\strseq{G})| = k$ for some $k > \arity{R} = 1$, then
    \[
        c = \frac{1}{1 + (k-1)\stonepar{R}{\strseq{A}}} \in (0,1)
        .
    \]
    Therefore, all the assumptions of Theorem~\ref{thm:fo_general_conditioning_with_infinitely_many_r_edges} are satisfied and $\strseq{A}*\strseq{G}$ is local convergent.
\end{proof}

\subsubsection{Proof of Theorem~\ref{thm:general_local_convergence}}\label{sssec:proof_of_general_local_convergence}

We carry out the proof in two steps.
First, we show that the $(k,r,p)$-representation equivalence is a refinement of the $k$-elementary equivalence on structures $(\str{A} * \str{G}, \tpl{a})^r$.
Consequently, the sum of sizes of representation equivalence classes yields the size of an elementary equivalence class.
Hence, it is enough to prove that the statistics of $(k,r,p)$-representation equivalence converge.
Therefore, as the second step, we show how the local statistics of $\str{A}$ and $\str{G}$ affect the statistics of $(k,r,p)$-representation equivalence of $p$-tuples in $\str{A}*\str{G}$.

Let us start with a simple lemma about elementary equivalence after a restriction to neighborhoods.

\begin{lemma}\label{lem:restriction_to_neighborhoods}
    Let $\lambda$ be a language with constants and suppose we have positive integers $k,r,t$ satisfying $t \geq r \cdot \maxarity{\lambda} + k$.
    Let $\str{A}, \str{B}$ be $\lambda$-structures such that Duplicator wins $\EF_t(\str{A}; \str{B})$.
    Then Duplicator also wins the game $\EF_k(\str{A}^r; \str{B}^r)$.
\end{lemma}
\begin{proof}
    We write $H = \EF_t(\str{A}; \str{B})$ and $H' = \EF_k(\str{A}^r; \str{B}^r)$.
    As usual, we use the game $H$ to determine the Duplicator's moves in the game $H'$.
    It is enough to verify that whenever Spoiler picks a vertex in the $r$-neighborhood of a constant in $\str{A}$ in the game $H$, Duplicator's response lies in the $r$-neighborhood of the constant in $\str{B}$ (and vice versa). 
    This follows from Lemma~\ref{lem:measuring_distances} as it is possible to measure distances up to $r$ for at least $k$ rounds of $H$ due to the assumption $t \geq r \cdot \maxarity{\lambda} + k$.
\end{proof}

We follow with the refinement property.

\begin{lemma}\label{lem:representation_equivalence_refines_elementary_equivalence}
    For $p$-tuples $\tpl{a}_1$ and $\tpl{a}_2$ from $\str{A}_1*\str{G}_1$ and $\str{A}_2*\str{G}_2$ with the profile $(I,E_1, \dots, E_t)$ we have
    \[
        (\str{A}_1*\str{G}_1, \tpl{a}_1) \approx_k^r (\str{A}_2*\str{G}_2, \tpl{a}_2)
        \implies
        (\str{A}_1 * \str{G}_1, \tpl{a}_1)^r \equiv_k (\str{A}_2 * \str{G}_2, \tpl{a}_2)^r
    \]
\end{lemma}
\begin{proof}
    For $i \in \{1,2\}$ in parallel, we iteratively apply Lemma~\ref{lem:continuity_bound_via_ef_games} to replace all the marked $R$-edges $\rho(a_\ell)$ in the base structures $\str{A}_i(\tpl{a}_i, r)$ by gadgets $\str{G}_i^j(\tpl{a}_i, r)$ for $j \in [t]$, where $\ell \in E_j$ (we keep the constants marking the external vertices of $\tpl{a}_i$).
    The remaining $R$-edges are replaced by the gadgets $\str{G}_i^r$.
    Denote the resulting structures by $\str{B}_1$ and $\str{B}_2$; observe that $\str{B}_i^r$ is isomorphic to the structure $(\str{A}_i*\str{G}_i, \tpl{a}_i)^r$ (possibly up to renaming constants).
    
    It remains to verify that $\str{B}_1^r \equiv_k \str{B}_2^r$: 
    we have started with $\str{A}_1(\tpl{a}_1, r) \equiv_{f(k,r,p)} \str{A}_2(\tpl{a}_2, r)$ and each application of Lemma~\ref{lem:continuity_bound_via_ef_games} reduces the degree of elementary equivalence by the factor of $\arity{R}$.
    Therefore, we have $\str{B}_1 \equiv_{r \cdot \maxarity{\lambda} + k} \str{B}_2$ due to our choice of the function $f$.
    The relation $\str{B}_1^r \equiv_k \str{B}_2^r$ follows from Lemma~\ref{lem:restriction_to_neighborhoods}.
\end{proof}

Now we proceed to show how to compute with the representation equivalence.
Let $\mathcal{C}$ be a class of $\approx_k^r$ assuming the the profile $\pi = (I,E_1, \dots, E_t)$ with a representative $(\str{A}_0*\str{G}_0, \tpl{a}_0)$, i.e. $(\str{A} * \str{G}, \tpl{a}) \in \mathcal{C}$ if and only if $(\str{A}_0 * \str{G}_0, \tpl{a}_0) \approx_k^r (\str{A} * \str{G}, \tpl{a})$.
The definition of $\approx_k^r$ implies that the class $\mathcal{C}$ may be described by $r$-constant-local formulas $\phi(\tpl{x}_1, \dots, \tpl{x}_p) \in \FOloc{|I|+(p-|I|)\arity{R}}(L_R)$, $\psi_j(x_1, \dots, x_{|E_j|}) \in \FOcloc{|E_j|}(L_G)$ for $j \in [t]$, and $\psi \in \FOcloc{0}(L_G)$ that capture the respective classes of $f(k,r,p)$-elementary equivalence of the structures $\str{A}_0(\tpl{a}_0, r)$, $\str{G}_0^j(\tpl{a}_0, r)$, and $\str{G}^r_0$.

We want to express the probability that for a $p$-tuple $\tpl{a}$ uniformly selected from $\str{A}*\str{G}$ holds $(\str{A}*\str{G}, \tpl{a}) \in \mathcal{C}$.

Recall that $\stonepar{\phi|_\pi}{\str{A}}$ denotes the probability that a uniformly selected sequence (of tuples) $\tpl{b}_1, \dots, \tpl{b}_p$ from $\str{A}$ satisfies $\phi$, given that it is a representation of a $p$-tuple with the profile $\pi$.
Furthermore, let $\stonepar{\psi_j|_{\nonroot{}}}{\str{G}}$ stand for the probability that $\str{G} \models \psi_j(c_1, \dots, c_{|E_j|})$ for a random $|E_j|$-tuple $\tpl{c}$, conditioned on the fact that each $c_i$ is a non-root of $\str{G}$.

The following statement summarizes the discussion and notation from above.

\begin{lemma}\label{lem:representation_statistics_determines_statistics}
    Fix a class $\mathcal{C}$ of $\approx_k^r$ as above.
    Let $\tpl{a}$ be a random $p$-tuple from $\str{A}*\str{G}$.
    Denote by $c$ the proportion of internal vertices in $\str{A}*\str{G}$ and by $m$ the number of $R$-edges in $\str{A}$.
    
    Then the probability that $\tpl{a}$ has profile $\pi$ is   
    \[
        c^{|I|}(1-c)^{p-|I|}\frac{m(m-1)\dots(m-t+1)}{r^{p-|I|}}
        .
    \]
    Given that $\tpl{a}$ has profile $\pi$, the probability that $(\str{A}*\str{G}, \tpl{a}) \in \mathcal{C}$ can be expressed as
    \[
        \stonepar{\phi|_\pi}{\str{A}} \cdot \stonepar{\psi}{\str{G}} \cdot \prod_{j \in [t]} \stonepar{\psi_j|_{\nonroot}}{\str{G}}
        .
    \]
\end{lemma}
\begin{proof}
    The calculation is straightforward.
    A random $p$-tuple $\tpl{a}$ has the internal and external vertices at the prescribed indices with the probability $c^{|I|}(1-c)^{p-|I|}$.
    The factor $\frac{m(m-1)\dots(m-t+1)}{r^{p-|I|}}$ calculates the probability that the external vertices are grouped in $t$ distinct copies of the gadget in $\str{A}*\str{G}$ according to the profile $\pi$.
    
    In the second part, the event $(\str{A}*\str{G}, \tpl{a}) \in \mathcal{C}$ occurs if and only if the structures $\str{A}(\tpl{a}, r)$ and each $\str{G}^j(\tpl{a}, r)$ satisfy the formulas $\phi$ and $\psi_j$, respectively (and $\str{G} \models \psi$, which does not depend on $\tpl{a}$).
    This probability is given by $\stonepar{\phi|_\pi}{\str{A}}$, resp. $\stonepar{\psi_j|_{\nonroot}}{\str{G}}$, and all these events are independent.
\end{proof}

Now we are ready to prove Theorem~\ref{thm:general_local_convergence}.

\begin{proof}[Proof of Theorem~\ref{thm:general_local_convergence}]
    Fix arbitrary $k, r \in \N$.
    We use Lemma~\ref{lem:representation_statistics_determines_statistics} to show that for each class $\mathcal{C}$ of $(k,r,p)$-representation equivalence for a non-trivial profile $\pi$ holds that the probabilities $P_n$ of $(\str{A}_n * \str{G}_n, \tpl{a}_n) \in \mathcal{C}$ converge (for trivial profiles the probability is $0$ as they do not occur a.a.s.).
    The probability $P_n$ is expressed as a finite product of probabilities; hence, the claim reduces to showing convergence of each factor, which follows directly from the assumptions.
    In particular, if the gadgets $\strseq{G}$ does eventually contain non-roots, the probability $\stonepar{\psi_j|_{\nonroot}}{\strseq{G}}$ converges (otherwise, this factor does not appear for a non-trivial profile).
    
    Let $\xi \in \FOloc{p}$ be $r$-local with $\qrank{\xi} \leq k$.
    Using the fact that $\approx_k^r$ is a refinement of $\equiv_k$ on $r$-neighborhoods (Lemma~\ref{lem:representation_equivalence_refines_elementary_equivalence}) of finite index, we may express the probability $\stonepar{\xi}{\strseq{A}*\strseq{G}}$ as a finite sum of convergent sequences.
    Thus, the sequence $\stonepar{\xi}{\strseq{A}*\strseq{G}}$ converges.
\end{proof}

\subsubsection{Generalization to multiple gadgets}\label{sssec:multiple_gadgets}

Here we generalize Theorem~\ref{thm:general_local_convergence} for repeated gadget construction with multiple gadgets.
This is preparation for Section~\ref{ssec:extension_to_fragmented_structures}, where we reduce the proof of local convergence of $\strseq{A}*\strseq{G}$ to showing local convergence of a sequence obtained by repeated application of gadget construction.

We consider a sequence $\strseq{A}$ in the language $L_\tpl{R} = L \cup \{R_1, \dots, R_\ell\}$ and sequences of gadgets $\strseq{G}^{(1)}, \dots, \strseq{G}^{(\ell)}$ for the respective symbols.
The language of each $\strseq{G}^{(j)}$ is $L$ extended by $\arity{R_j}$ constants for roots; in particular, $\strseq{G}^{(j)}$ contains no $R_k$ edges for any $k$.
We aim for local convergence of the sequence $\strseq{B} = (\dots(\strseq{A}*\strseq{G}^{(1)}) * \dots ) *\strseq{G}^{(\ell)}$.
We usually omit the parenthesis since the intended order of the evaluation is obvious.
The \emph{internal} vertices of $\strseq{B}$ are the original vertices of $\strseq{A}$ while the other vertices are \emph{external}.
Specifically, the external vertices in a copy of $\strseq{G}^{(j)}$ are \emph{$j$-external}.

Let us generalize the notion of a profile.

\begin{definition}[Multi-profile]
    Let $\tpl{a}$ be a $p$-tuple from $\str{B} = \str{A}*\str{G}^{(1)}* \dots * \str{G}^{(\ell)}$.
    The multi-profile of a $p$-tuple $\tpl{a}$ from $\str{B}$ is a partition $(I, \mathcal{E}^1, \dots, \mathcal{E}^\ell)$ of the set $[p]$, where each $\mathcal{E}^j$ is partitioned into $E^j_1, \dots, E^j_{t_j}$ such that
    \begin{align*}
        I             &= \{i : a_i\text{ is internal}\}
        , \\
        \mathcal{E}^j &= \{i : a_i\text{ is $j$-external}\}
        .
    \end{align*}
    Two indices $i,i' \in \mathcal{E}^j$ share a set $E^j_k$ if and only if $\rho(a_i) = \rho(a_{i'})$, where $\rho(a_i)$ denotes the $R_j$-edge of the gadget's copy where $a_i$ lie.
    The sets $I$ and $\mathcal{E}^j$ are possibly empty while we require each $E^j_k$ being non-empty.
    The sets $E^j_1, \dots, E^j_{t_j}$ are listed by the ascending order of their minimal elements.
\end{definition}

Again, we call a multi-profile $\pi$ trivial with respect to the sequence $\strseq{B}$ if the probability that a random $p$-tuple from $\strseq{B}$ has the multi-profile $\pi$ tends to $0$.

We also revise the symbol $\stonepar{\phi|_\pi}{\str{A}}$, where $\pi$ is a multi-profile of a $p$-tuple and $\phi(\tpl{x}_1, \dots, \tpl{x}_p) \in L_\tpl{R}$ is a formula with $p$ blocks of free variables.
We write $\stonepar{\phi|_\pi}{\str{A}}$ for the probability that $\phi$ is satisfied by a uniformly chosen vertices $\tpl{b}_1, \dots, \tpl{b}_p$ from $\str{A}$, given that $\tpl{b}_i \in R_j^\str{A}$ if $i \in \mathcal{E}^j$, and for $i, i' \in \mathcal{E}^j$ it holds that $\tpl{b}_i = \tpl{b}_{i'}$ iff $i, i' \in E^j_k$ for some $k$.

\begin{theorem}\label{thm:local_multiple_gadgets}
    Fix $p \in \N$.
    Let $\strseq{A}$ be a sequence of base structures and $\strseq{G}^{1}, \dots, \strseq{G}^{\ell}$ be sequences of gadgets.
    Write $\strseq{B}$ for the sequence $\strseq{A}*\strseq{G}^{(1)} * \dots *\strseq{G}^{(\ell)}$.
    Suppose that the following conditions hold:
    \begin{enumerate}[(i)]
        \item for every multi-profile $\pi = (I, \mathcal{E}^1, \dots, \mathcal{E}^\ell)$ of a $p$-tuple that is non-trivial w.r.t. $\strseq{B}$ holds that for each $\phi$ with $p$ blocks the sequence $\stonepar{\phi|_\pi}{\strseq{A}}$ converges,
        \item for each $j \in [\ell]$ is $\strseq{G}^{(j)}$ an $\FOcloc{m_j}$-convergent sequence of gadgets, where
        \[
            m_j = \max \{ |E^j_k| : \pi = (I, \mathcal{E}^1, \dots, \mathcal{E}^\ell) \text{ is non-trivial w.r.t. } \strseq{B} \}
        \]
        \item the proportion of internal vertices and $j$-external vertices, for each $j \in [\ell]$, in $\strseq{B}$ tends to a limit.
    \end{enumerate}
    Then the sequence $\strseq{B}$ is $\FOloc{p}$-convergent.
\end{theorem}
\begin{proof}
    We proceed by induction on $\ell$.
    For $\ell = 1$, the statement reduces to Theorem~\ref{thm:general_local_convergence}.
    
    Consider $\ell > 1$.
    Our plan is to apply gadget construction once to obtain the sequence $\strseq{C} = \strseq{A}*\strseq{G}^{(\ell)}$ and then use the induction hypothesis for the base structures $\strseq{C}$ and gadgets $\strseq{G}^{(1)}, \dots, \strseq{G}^{(\ell - 1)}$.
    We only need to verify that all the conditional probabilities $\stonepar{\phi|_{\pi}}{\strseq{C}}$ converge.
    This, in fact, follows by a refinement of the ideas behind the representation equivalence technique.
    
    Let $\pi = (I, \mathcal{E}^1, \dots, \mathcal{E}^{\ell-1})$ be a non-trivial multi-profile of a $p$-tuple from $\strseq{C}*\strseq{G}^{(1)}*\dots*\strseq{G}^{(\ell-1)}$.
    The set $I$ contains indices of vertices of $\strseq{C}$ that divide into internal and $\ell$-external.
    Thus, we can decompose the set $I$ into $I'$ and $\mathcal{E}^{\ell}$ obtaining a profile $\pi' = (I', \mathcal{E}^1, \dots, \mathcal{E}^{\ell-1}, \mathcal{E}^{\ell})$ for the structure $\strseq{A}*\strseq{G}^{(1)} * \dots *\strseq{G}^{(\ell)}$.
    
    Consider tuples $\tpl{b}_1, \dots, \tpl{b}_p$ from $\str{C}_n$ representing a tuple with the multi-profile $\pi$.
    We can infer whether $\str{C}_n \models \phi(\tpl{b}_1, \dots, \tpl{b}_p)$ from (the behavior of) the representation of $\tpl{b}_1, \dots, \tpl{b}_p$ in $\str{A}_n$, which always has one of the multi-profiles $\pi'$, and the representation in (copies of) $\str{G}^{(\ell)}_n$.
    As $\pi'$ is certainly a non-trivial profile w.r.t. $\strseq{B}$, the sequence $\stonepar{\phi'|_{\pi'}}{\strseq{A}}$ converges for any formula $\phi'$ with $p$ blocks that describes the behavior the representation of $\tpl{b}_1, \dots, \tpl{b}_p$ in $\str{A}_n$.
    
    It remains to observe that the representation $\tpl{b}_1, \dots, \tpl{b}_p$ in $\strseq{A}$ have one particular profile $\pi'$ with a convergent probability, i.e. that the probability of a vertex from $\strseq{C}$ being an internal vertex or an $\ell$-external vertex converges.
    This is the assumption from the statement.
\end{proof}

\subsection{Exploiting locality}\label{ssec:exploiting_locality}

Here we state another kind of a sufficient condition for local convergence, which does not involve the assumption on conditional behavior of $R$-edges in $\strseq{A}$.
We prove that if the mass of the gadgets around the roots is vanishing, the structures $\str{A}_n*\str{G}_n$ behave essentially the same as the disjoint union of the structures $\str{A}_n$ endowed by neighborhoods of the roots of $\str{G}_n$ and $|R^{\str{A}_n}|$ copies of $\str{G}_n$.
Such a decomposition trivially implies the local convergence of the sequence $\strseq{A}*\strseq{G}$ provided that the sequence $|R^\strseq{A}|$ and the proportion of internal vertices tends to a limit.

We start with a more general treatment based on \cite{clustering} that justifies our approach.
Recall that $\nu_\str{A}(S)$ denotes the relative size of the set $S \subseteq V(\str{A})$ within the structure $\str{A}$.

\begin{definition}\label{def:small_and_negligible}
    Let $\lambda$ be a relational language with constants.
    Let $\strseq{A}$ be a sequence of $\lambda$-structures and let $\seq{S}$ be a sequence of subsets of $\strseq{A}$, i.e. $S_n \subseteq V(\str{A}_n)$.
    The sequence $\seq{S}$ is \emph{negligible} if 
    \[
        \forall r \in \N: \limsup \nu_\strseq{A}(N_\strseq{A}^r(\seq{S})) = 0
        .
    \]
    Moreover, a negligible sequence $\seq{S}$ is \emph{strongly-negligible} if each $r$-neighborhood of $\seq{S}$ eventually avoids all constants in $\strseq{A}$.
    That is, if
    \[
        \forall r \in \N \;\exists n_0 \in \N \;\forall n \geq n_0: \text{ the set $N_{\str{A}_n}^r(S_n)$ contains no constant from $\str{A}_n$}
        .
    \]
\end{definition}

Two sequences $\strseq{A}$ and $\strseq{B}$ that differ only by a negligible sequence are called \emph{equivalent}.
If $\strseq{A}$ and $\strseq{B}$ are equivalent and $\strseq{A}$ is local convergent, then $\strseq{B}$ is also local convergent with $\lim \stonepar{\phi}{\strseq{A}} = \lim \stonepar{\phi}{\strseq{B}}$ for each $\phi \in \FOloc{}$ \cite[Lemma~3]{clustering}.
Our variant of the notion behaves analogously with respect to the constant-local formulas and constant-local convergence.
If two sequences $\strseq{A}$ and $\strseq{B}$ differ only by a strongly-negligible sequence, we called them \emph{strongly-equivalent}.

Strongly-equivalent sequences of gadgets are a natural tool for examining local convergence of sequences of resulting structures.

\begin{lemma}\label{lem:negligible_sequences_and_equivalent_structures}
    Let $\strseq{A}$ be a sequence of base structures and $\strseq{G}, \strseq{G}'$ be strongly-equivalent sequences of gadgets.
    Then $\strseq{A}*\strseq{G}$ is equivalent to $\strseq{A}*\strseq{G}'$.
\end{lemma}
\begin{proof}
    The sequences $\strseq{A}*\strseq{G}$ and $\strseq{A}*\strseq{G}'$ differ only by a union of negligible sequences which is a negligible sequence (the size of the union is negligible w.r.t. the total size of gadget's copies).
\end{proof}

There is an obvious way how to turn a negligible sequence into a strongly-negligible sequence.
For a function $f: \N \to \N$, the expression $\strseq{A}^{\bm{f}}$ stands for the sequence $(\str{A}_n^{f(n)})_{n \in \N}$.

\begin{lemma}\label{lem:from_negligible_to_strongly_negligible}
    Let $\seq{S}$ be a negligible sequence in $\strseq{A}$ and $f: \N \to \N$ be a non-decreasing unbounded function.
    Then the sequence $\seq{S}' = \seq{S} \setminus V(\strseq{A}^{\bm{f}})$ is strongly-negligible.
\end{lemma}
\begin{proof}
    The sequence $\seq{S}' \subseteq \seq{S}$ is obviously negligible.
    Moreover, it is strongly-negligible as we actively remove from the sequence the neighborhood of all constants of (eventually) arbitrarily large radius.
\end{proof}

We usually want to choose a slowly growing function $f$, otherwise it may happen that the sequence $\seq{S}'$ is a sequence of empty sets.

We state two standard facts about disjoint unions of local convergent sequences.
All the structures are $\lambda$-structures for a purely relational language $\lambda$.

\begin{fact}[\cite{clustering}, by Corollary 3]\label{fact:union_of_convergent_sequences}
    Let $\strseq{A}$ be a \emph{stable} disjoint union of local convergent sequences $\strseq{A}_1, \dots, \strseq{A}_n$ in the sense that for each $i \in [n]$ the limit $c_i$ of $\nu_\strseq{A}(V(\strseq{A}_i))$ exists.
    Then $\strseq{A}$ is local convergent.
\end{fact}

\begin{fact}[\cite{modeling_limits_in_hereditary_classes}, by Lemma 17]\label{fact:union_of_infinitely_many_copies}
    Let $\strseq{A}$ be $\FOloc{1}$-convergent sequence and $f: \N \to \N$ be a function with $f \to \infty$.
    Suppose that $\str{B}_n$ is the disjoint union of $f(n)$ copies of $\str{A}_n$.
    Then the sequence $\strseq{B}$ is local convergent.
\end{fact}

Let us follow with the main result of this part.
Recall that $\partial_\str{A} S$ stands for the set $N_\str{A}(S) \setminus S$.

\begin{theorem}\label{thm:fo_small_roots}
    Let $\strseq{A}$ be a local convergent sequence of base structures and $\strseq{G}$ be a constant-local convergent sequence of gadgets satisfying
    \begin{enumerate}[(i)]
        \item the sequence $\tpl{z}^\strseq{G}$ of roots is negligible in $\strseq{G}$,
        \item the proportion of internal vertices in $\strseq{A}*\strseq{G}$ tends to $c$,
        \item $\lim |R^\strseq{A}|$ exists.
    \end{enumerate}
    Then the sequence $\strseq{A}*\strseq{G}$ is local convergent.
\end{theorem}
\begin{proof}
    Fix a function $f:\N \to \N$ satisfying that the sequence $V(\strseq{G}^{\bm{f}})$ is negligible in $\strseq{G}$.
    For example, define $f,g: \N \to \N$ as follows:
    \begin{align*}
        g(r) &= \min 
        \left\lbrace 
            n \;\middle|\; \forall m \geq n : \nu_{\str{G}_m}(V(\str{G}_m^r)) < 2^{-r}
        \right\rbrace
        , \\
        f(n) &= \frac{1}{2} \min\{ 
        \sup 
        \left\lbrace 
            r \;\middle|\; g(r) < n
        \right\rbrace
        ,
        n
        \}
        .
    \end{align*}
    The non-decreasing function $g$ is well defined as $\tpl{z}^\strseq{G}$ is a negligible set in $\strseq{G}$, the function $f$ is non-decreasing and unbounded.
    Observe that the sequence $V(\strseq{G}^{\bm{f}})$ is negligible in $\strseq{G}$: for $r \in \N$ we eventually have $V(\strseq{G}^{\bm{f}+r}) \subseteq V(\strseq{G}^{2\bm{f}})$ and $\nu_\strseq{G}(V(\strseq{G}^{2\bm{f}})) \to 0$ by the choice of $f$.
    
    Set $\seq{S} = \partial_\strseq{G} V(\strseq{G}^{\bm{f}})$ and $\strseq{G}' = \strseq{G} \setminus \seq{S}$.
    The sequence $\seq{S}$ is strongly negligible by Lemma~\ref{lem:from_negligible_to_strongly_negligible}.
    Therefore, by Lemma~\ref{lem:negligible_sequences_and_equivalent_structures}, we only need to prove local convergence of the sequence $\strseq{A}*\strseq{G}'$.
    To do so, we decompose the sequence into a stable disjoint union of sequences $\strseq{B}$ and $\strseq{C}$ and use Fact~\ref{fact:union_of_convergent_sequences}.
    
    Our choice of $\seq{S}$ decomposes $\strseq{G}'$ into the stable disjoint union of a sequence of gadgets $\strseq{H} = \strseq{G}^{\bm{f}}$ and a sequence of $L$-structures $\strseq{K} = \strseq{G}'_n \setminus V(\str{H}_n)$.
    We write $\strseq{B}$ for the sequence $\strseq{A}*\strseq{H}$ and $\strseq{C}$ for the sequence of disjoint unions of $|R^\strseq{A}|$ copies of $\strseq{K}$.
    Observe that their proportion is stable and follows the proportion of internal vertices in $\strseq{A}*\strseq{G}$, i.e. $\lim \nu_{\strseq{A}*\strseq{G}'}(V(\strseq{C})) = 1 - c$ as $\strseq{C}$ contain the dominant portion of external vertices from $\strseq{A}*\strseq{G}'$.
    
    If $c = 0$, it is enough to prove local convergence of $\strseq{C}$ ($\strseq{B}$ is negligible in $\strseq{A}*\strseq{G}'$).
    That follows from Fact~\ref{fact:union_of_convergent_sequences}~or~\ref{fact:union_of_infinitely_many_copies} as $\strseq{C}$ the disjoint union of $|R^\strseq{A}|$ copies of a local convergent sequence (depending whether $\lim |R^\strseq{A}|$ is finite or not).
    
    If $c > 0$, we need to additionally prove local convergence of $\strseq{B}$.
    Observe that sequence of gadgets $\strseq{H}$ is $\FOcloc{0}$-convergent as $f$ is non-decreasing (in fact, it suffices that $f$ has a limit).
    Moreover, the proportion of internal vertices in $\strseq{B} = \strseq{A}*\strseq{H}$ tends to $1$ (only the internal vertices may account for $\lim \nu_{\strseq{A}*\strseq{G}'}(V(\strseq{B})) = c$).
    Therefore, the local convergence of $\strseq{B}$ follows from Theorem~\ref{thm:fo_dominant_internal_vertices}.
\end{proof}

In certain cases, we may omit some assumptions.
If $c = 0$, we do not need the convergence of $\strseq{A}$.
If $\lim |R^\strseq{A}| = \infty$, it is enough to assume $\FOcloc{1}$-convergence of $\strseq{G}$, which implies $\FOloc{}$-convergence of the sequence $\strseq{C}$ by Fact~\ref{fact:union_of_infinitely_many_copies}.

\subsection{Extension to fragmented structures}\label{ssec:extension_to_fragmented_structures}

Here we combine the previous approaches for obtaining local convergence with the idea of fragmentation from Section~\ref{ssec:fragmentation}.
We consider the sequence $\strseq{A}^\sigma$, where we fragment the $R$-edges into subedges according to limit distances of roots in $\strseq{G}$.
We define \emph{tip} of $\strseq{G}$ as the sequence of roots from a single class of $\sigma$.
We distinguish \emph{light} tips that form a negligible sequence in $\strseq{G}$ and \emph{heavy} tips that do not.
Using Theorem~\ref{thm:local_multiple_gadgets}, we show that a sufficient assumption for local convergence of $\strseq{A}*\strseq{G}$ is convergence of conditional probabilities, but \emph{only} those where we condition on selection of subedges corresponding to heavy tips.

Fix a sequence $\strseq{A}$ of base structures and a constant-local convergent sequence $\strseq{G}$ of gadgets for the rest of the section.
Let $\sigma$ be the equivalence on $[\arity{R}]$ from~\eqref{eq:canonical_sigma} with classes $X_0, \dots, X_\ell$, where $X_0$ is the empty class.
(Note that constant-local sentences are sufficient for the definition of $\sigma$.)

\begin{definition}[Tips]
    Let $X$ be a class of $\sigma$.
    We call the sequence $\tpl{z}^\strseq{G}_X = \{z^\strseq{G}_i, i \in X\}$ a \emph{tip} of $\strseq{G}$.
    For $X \not= X_0$, the tip $\tpl{z}^\strseq{G}_X$ is \emph{light} if $\tpl{z}^\strseq{G}_X$ is a negligible sequence in $\strseq{G}$.
    Otherwise, we say that the tip is \emph{heavy}.
\end{definition}

We show that there is a decomposition of $\strseq{G}$ into a strongly-negligible sequence $\seq{S}$ and a disjoint union of sequences $\strseq{G}^{(0)}, \dots, \strseq{G}^{(\ell)}$ with the following properties for each $i \in [\ell]_0$:
\begin{enumerate}[(i)]
    \item the sequence $\strseq{G}^{(i)}$ eventually contains the tip $\tpl{z}^\strseq{G}_{X_i}$,
    \item the limit of $\nu_\strseq{G}(V(\strseq{G}^{(i)}))$ exists and is $0$ if $\tpl{z}^\strseq{G}_{X_i}$ is a light tip,
    \item if $\lim \nu_\strseq{G}(V(\strseq{G}^{(i)})) > 0$, the sequence $\strseq{G}^{(i)}$ is constant-local convergent.
\end{enumerate}
We call such a decomposition of $\strseq{G}$ a \emph{good clustering} of $\strseq{G}$.
Note that the first condition together with strong-negligibility of $\seq{S}$ implies that arbitrarily large neighborhood $N^r_\strseq{G}(\tpl{z}^\strseq{G}_{X_i})$ eventually lies in $\strseq{G}^{(i)}$.
Moreover, the strong-negligibility implies $\FOcloc{0}$-convergence of $\strseq{G}^{(i)}$ from the third condition; hence, only local convergence needs to be proved.

\begin{lemma}\label{lem:existence_of_good_clustering}
    There exists a good clustering of $\strseq{G}$.
\end{lemma}

We leave the proof of this key statement for Section~\ref{sssec:proof_of_good_clustering}, we first show how to use the decomposition.
Write $\strseq{G}'$ for the sequence $\strseq{G} \setminus \seq{S}$, i.e. the union of all $\strseq{G}^{(i)}$.
By Lemma~\ref{lem:negligible_sequences_and_equivalent_structures}, it is enough to prove local convergence of $\strseq{A}*\strseq{G}'$ to obtain local convergence of $\strseq{A}*\strseq{G}$.

Write $\strseq{H}^{(j)}$ for the sequence of gadgets $\strseq{G}^{(j)}$ where we add an \emph{auxiliary} root, an isolated vertex, to each structure.
Observe that $\strseq{B} = \strseq{A}^\sigma*\strseq{H}^{(0)}*\dots*\strseq{H}^{(\ell)}$ is isomorphic to the union of the structures $\strseq{A}*\strseq{G}'$ and a sequence of independent sets $\seq{I}$ of size $(\ell+1)|R^\strseq{A}|$.
Provided that the gadgets $\strseq{G}$ eventually contain at least one non-root, which is the non-trivial case, we have $\lim \nu_\strseq{B}(\seq{I}) < 1$.
Thus, removal of the sequence $\seq{I}$ does not harm the local convergence.
Consequently, obtaining local convergence of the sequence $\strseq{A}*\strseq{G}$ reduces to obtaining local convergence of the sequence $\strseq{A}^\sigma*\strseq{H}^{(0)}*\dots*\strseq{H}^{(\ell)}$.
Sufficient conditions are found in the statement of Theorem~\ref{thm:local_multiple_gadgets}.

As a result, we have the following corollary.
Note that a multi-profile $\pi = (I, \mathcal{E}^1, \dots, \mathcal{E}^\ell)$ with $|\mathcal{E}^j| > 0$ for any light tip $\tpl{z}^\strseq{G}_{X_j}$ is trivial as $\nu_\strseq{G}(V(\strseq{G}^{(j)})) = 0$.

\begin{corollary}\label{cor:loc_fragmented_structures} 
    Let $\strseq{A}$ be a sequence of base structures and $\strseq{G}$ be a constant-local convergent sequence of gadgets inducing an equivalence $\sigma$ on $[\arity{R}]$ satisfying
    \begin{enumerate}[(i)]
        \item for every $p \in \N$ and a multi-profile $\pi = (I, \mathcal{E}^1, \dots, \mathcal{E}^\ell)$ of a $p$-tuple that is non-trivial w.r.t. $\strseq{A}^\sigma*\strseq{H}^{(0)}*\dots*\strseq{H}^{(\ell)}$ holds that for each $\phi$ with $p$ blocks the sequence $\stonepar{\phi|_\pi}{\strseq{A}^\sigma}$ converges,
        \item the proportion of internal vertices in $\strseq{A}*\strseq{G}$ tends to a limit.
    \end{enumerate}
    Then the sequence $\strseq{A}*\strseq{G}$ is local convergent.
\end{corollary}

\subsubsection{Proof of Lemma~\ref{lem:existence_of_good_clustering}}\label{sssec:proof_of_good_clustering}

We construct the decomposition in two steps.
First, we use the result form \cite{clustering} to find an initial negligible sequence $\seq{Y}$ that works well with heavy tips.
Then, we modify it to also accommodate the light tips.

\paragraph{Clustering in local convergent sequences}\label{par:clustering_in_local_convergent_sequences}

Let us survey the key definitions and results from \cite{clustering} that we are going to use.
Some parts of the text are verbatim transcriptions with only minor modifications.

Let $\lambda$ be a purely relational language.
Let $\strseq{A}$ be a local convergent sequence of $\lambda$-structures and let $\seq{Z}$ be a sequence of subsets of $\strseq{A}$.
We denote by $\lift{\seq{Z}}{\strseq{A}}$ the lift of $\strseq{A}$ obtained by marking all elements of sets $\seq{Z}$ by a new unary symbol.
The sequence $\seq{Z}$ is a \emph{cluster} if $\lift{\seq{Z}}{\strseq{A}}$ is local convergent and the boundary of $\seq{Z}$, the set $\partial_\strseq{A} \seq{Z}$, is a negligible set.
In particular, $\seq{Z}$ is a \emph{globular cluster} if it is not a negligible sequence and for $\strseq{Z}$, the sequence of substructures induced by $\seq{Z}$, we have that for every $\eps > 0$ there is $r \in \N$ such that
\[
    \liminf_{n \to \infty} \sup_{v_n \in \str{Z}_n} \nu_{\str{Z}_n}(N^r_{\str{Z}_n}(v_n)) > 1-\eps
    .
\]
That is, the mass of $\seq{Z}$ is strongly concentrated around a single point.
On the other hand, cluster $\seq{X}$ is \emph{residual} if it contains no point with a positive mass in its neighborhood: 
if for all $r \in \N$ holds
\[
    \limsup_{n \to \infty} \sup_{v_n \in Z_n} \nu_{\str{A}_n}(N^r_{Z_n}(v_n)) = 0
    .
\]
If $\seq{Z}$ is a cluster in $\strseq{A}$, the sequence $\nu_\strseq{A}(\seq{Z})$ tends to a limit.
If the limit is positive, the sequence $\strseq{Z}$ is local convergent.

Let $\strseq{A}$ be a local convergent sequence of $\lambda$-structures.
A lifted sequence $\lift{}{\strseq{A}}$ of $\strseq{A}$ obtained by extending the language $\lambda$ into $\lambda^+$ by adding countably many unary symbols $M_1, M_2, \dots$ is a \emph{clustering} if, denoting
\[
    \seq{Y} = V(\strseq{A}) \setminus \bigcup_i M_i^\strseq{A}
    ,
\]
the following conditions holds:
\begin{enumerate}[(i)]
    \item the sequence $\lift{}{\strseq{A}}$ is local convergent,
    \item the sequence $\seq{Y}$ is negligible and $\bigcup_i \partial_\strseq{A} M_i^\strseq{A} \subseteq \seq{Y}$,
    \item for every $n \in \N$ the non-empty sets among $Y_n, M_1^{\str{A}_n}, M_2^{\str{A}_n}, \dots$ form a partition of $\str{A}_n$.
    \item The partition is stable in the sense that
    \[
        \sum_i \lim \stonepar{M_i}{\strseq{A}} = \lim \sum_i \stonepar{M_i}{\strseq{A}}
        .
    \]
\end{enumerate}

The definition implies that each marked sequence $M_i^\strseq{A}$ is a cluster.

The main result of the paper (stated here in a weaker form) is the following detection of globular clusters.

\begin{theorem}[\cite{clustering}, Theorem~1]\label{thm:clustering}
    Let $\strseq{A}$ be a local-convergent sequence of $\lambda$-structures.
    Then there is an extended language $\lambda^+ = \lambda \cup \{M_R, M_i, i \in \N\}$ and a clustering $\strseq{A}^+$ of $\strseq{A}$ with the following properties:
    \begin{enumerate}[(i)]
        \item for every $i \in \N$ the sequence $M_i^{\strseq{A}}$ is a globular cluster,
        \item $M_R^{\strseq{A}}$ is a residual cluster.
        \item the unmarked vertices form a negligible sequence.
    \end{enumerate}
\end{theorem}

We call the clustering from Theorem~\ref{thm:clustering} a \emph{globular clustering}.

\paragraph{Clustering in gadgets}\label{par:clusteting_in_gadgets}

We start with a globular clustering of $\strseq{G}$.
Strictly speaking, Theorem~\ref{thm:clustering} assumes a purely relational language while the sequence $\strseq{G}$ contains roots; however, we can replace them by unary marks.
We proceed to show that the globular clustering interacts well with the heavy tips.

\begin{lemma}\label{lem:default_clustering_of_gadgets}
    Let $\strseq{G}^+$ be a globular clustering and let $\tpl{z}^\strseq{G}_X$ be a heavy tip.
    Then there is a globular cluster marked by a symbol $M$ such that for each $r \in \N$ eventually $N_\strseq{G}^r(\tpl{z}^\strseq{G}_X) \subseteq M^{\strseq{G}^+}$.
    Moreover, the clusters for different heavy tips are distinct.
\end{lemma}
\begin{proof}
    The tip $\tpl{z}^\strseq{G}_X$ is heavy, so there is $d \in \N$ such that $\lim \nu_\strseq{G}(N^d_\strseq{G}(\tpl{z}^\strseq{G}_X)) > 0$.
    As the total mass of clusters tends to $1$, it eventually holds that almost all vertices from the sets $N^d_\strseq{G}(\tpl{z}^\strseq{G}_X)$ lie in a marked cluster.
    Since the boundary of each cluster is a negligible sequence, the whole ball $N^d_\strseq{G}(\tpl{z}^\strseq{G}_X)$ lies eventually in the cluster (otherwise the $2d$-neighborhood of $\partial_\strseq{G} M^\strseq{G}$ has positive mass).
    The same reasoning applies to $N^r_\strseq{G}(\tpl{z}^\strseq{G}_X)$ for any $r \geq d$.
    
    The cluster cannot be residual, which is witnessed by the positive mass around the sequence $z_i^\strseq{G}$ for an arbitrary $i \in X$.
    Moreover, two different heavy tips cannot share a common globular cluster as they concentrate a positive mass of vertices in their $r$-neighborhoods for some fixed $r$, but tend away from each other.
    This is incompatible with the definition of the globular cluster.
\end{proof}

\begin{proof}[Proof of Lemma~\ref{lem:existence_of_good_clustering}]
    Let $\strseq{G}^+$ be a globular clustering.
    Consider a non-decreasing unbounded function $f: \N \to \N$ satisfying:
    \begin{enumerate}[(i)]
        \item $f(n) < \frac{1}{3}\min\{ \dist_{\str{G}_n}(z^{\str{G}_n}_i, z^{\str{G}_n}_j) : (i,j) \not\in \sigma \}$,
        \item if $\tpl{z}^\strseq{G}_X$ is a light tip, then $\nu_\strseq{G}(N^{2\bm{f}}_\strseq{G}(\tpl{z}^\strseq{G}_X)) \to 0$.
    \end{enumerate}
    Such a function can be constructed similarly as in the proof of Theorem~\ref{thm:fo_small_roots}.
    
    We obtain the desired clustering as follows:
    for each light tip $\tpl{z}^\strseq{G}_X$, we assign a new mark $M_X$ to the vertices $N^{\bm{f}}_\strseq{G}(\tpl{z}^\strseq{G}_X)$ and remove marks from $\partial N^{\bm{f}}_\strseq{G}(\tpl{z}^\strseq{G}_X)$.
    Observe that this is indeed a clustering since we only modify negligible sets; in particular, the vertices with removed marks form a negligible sequence thanks to $\nu_\strseq{G}(N^{2\bm{f}}_\strseq{G}(\tpl{z}^\strseq{G}_X)) \to 0$.
    Moreover, the sequence of unmarked vertices is now strongly-negligible as we have marked an (eventually) arbitrarily large neighborhood of light tips (using that $f \to \infty$).
    Note that arbitrarily large neighborhoods of heavy tips lie in a cluster by Lemma~\ref{lem:default_clustering_of_gadgets}.
    
    We finish the decomposition of by setting $\strseq{G}^{(i)}$ to be the structure induced by the cluster containing the tip $\tpl{z}^\strseq{G}_{X_i}$.
    Resp. for $X_0$, we set $\strseq{G}^{(0)}$ to be the union of all the remaining clusters.
    For light tips, we have that $\nu_\strseq{G}(V(\strseq{G}^{(i)})) \to 0$ by the choice of $f$.
    The rest of the second requirement and the third requirement for the decomposition are satisfied as the modified marks form a clustering.
\end{proof}

\section{Inverse theorems for local convergence}\label{sec:inverse_theorems}

This section is devoted to inverse theorems for local convergence of the sequence $\strseq{A}*\strseq{G}$, i.e. to statements of the form:
if $\strseq{A}*\strseq{G}$ is local convergent, then the sequences $\strseq{A}$ and $\strseq{G}$ satisfy some property.

Such a description seems difficult in general.
There are (easy to construct) examples of sequences $\strseq{A}$ and $\strseq{G}$ that do not even converge and still produce a convergent result.
Hence, we establish a stronger notion of convergence and restrict our attention to those sequences whose resulting sequence $\strseq{A}*\strseq{G}$ converge in this stronger sense.

To formalize the stronger notion of convergence, we revisit the definition of the structure $\str{A}*\str{G}$.
We introduce new symbols to the language of resulting structures to make certain important features of the structure $\str{A}*\str{G}$ definable.

\begin{definition}[Construction language]\label{def:construction_language}
    Define $L_* = L_R \cup \{\Int, \Ext, \rho\}$, where the relation symbols $\Int$ and $\Ext$ are unary, and $\rho$ is of arity $\arity{R}+1$.
\end{definition}

\begin{definition}[Gadget construction with construction language]\label{def:gadget_construction_with_construction_language}
    Let $\str{A}$ be a base structure and $\str{G}$ a gadget.
    Abusing notation, we denote by $\str{A}*\str{G}$ the $L_*$-lift of the structure $\str{A}*\str{G}$ from Definition~\ref{def:gadget_construction}.
    The additional symbols are interpreted as follows:
    $R^{\str{A} * \str{G}}$ marks the $R$-edges of $\strseq{A}$, the sets $\Int^{\str{A}*\str{G}}$ and $\Ext^{\str{A}*\str{G}}$ partition $V(\str{A}*\str{G})$ into internal and external vertices, and $\rho^{\str{A}*\str{G}}$ is the graph of the (partial) function $\rho$ that maps the external vertices to their corresponding $R$-edge. 
    
    In the set notation, we have:
    \begin{align*}
        R^{\str{A}*\str{G}}    &= \{([a_1], \dots, [a_s]) : (a_1, \dots, a_s) \in R^\str{A} \}
        ,\\
        \Int^{\str{A}*\str{G}} &= \{[x] : [x] \in V(\str{A} * \str{G}) \text{ is internal, i.e. contains a vertex of }\str{A} \}
        ,\\
        \Ext^{\str{A}*\str{G}} &= V(\str{A} * \str{G}) \setminus \Int^{\str{A}*\str{G}}
        ,\\
        \rho^{\str{A}*\str{G}} &= \{([(\tpl{e}, v)], [e_1], \dots, [e_\arity{R}]) : [(\tpl{e}, v)] \in \Ext^{\str{A}*\str{G}}\} \subseteq \Ext^{\str{A}*\str{G}} \times R^{\str{A}*\str{G}}
    \end{align*}
\end{definition}

In this section, we consider the resulting structures $\str{A}*\str{G}$ to be $L_*$-structures according to the definition above.
Clearly, $\FO(L_*)$-convergence implies $\FO(L)$-convergence as $L \subset L_*$.
\todo{Sometimes, we also have the converse: consider some rigid graph as a gadget with the property that $\str{A}*\str{G}$ allows to identify the copies of $\str{G}$ (e.g. the graph with three 7-cycles).}

Also, we are going to assume that no $S$-edge spans the roots of a gadget $\str{G}$.
In such a case, the substructure of $\str{A}*\str{G}$ induced by the set $\Int^{\str{A}*\str{G}}$ is isomorphic to $\str{A}$.
We remark that with a bit more care it is possible to determine the structure $\str{A}$ from $\str{A}*\str{G}$ under a milder assumption that no edge in $\str{A}*\str{G}$ has two sources, i.e. every edge either comes from $\str{A}$ or a copy of $\str{G}$ but not from both, provided that the positions of edges on gadgets' root are constant.

\subsection{Construction language}\label{ssec:construction_language}

Here we give an inverse theorem for $\FOloc{}(L_*)$-convergence of $\strseq{A}*\strseq{G}$.
This is the strongest sense of convergence that we consider.

\begin{theorem}\label{thm:inverse_for_construction_language}
    Suppose that $\strseq{A}*\strseq{G}$ is an $\FOloc{p}(L_*)$-convergent sequence.
    Then the conditional probabilities in the sense of Theorem~\ref{thm:general_local_convergence} converge.
    That is, for every $p \in \N$ and a non-trivial profile $\pi = (I,E_1, \dots, E_t)$ of a $p$-tuple holds that for each $\phi \in \FOloc{|I|+(p-|I|)\arity{R}}(L_R)$ the sequence $\stonepar{\phi|_\pi}{\strseq{A}}$ converges.
\end{theorem}
\begin{proof}
    Fix a non-trivial profile $\pi$ of a $p$-tuple and a formula $\phi \in \FOloc{}(L_R)$ with $p$ blocks of free variables.
    There is a formula $\psi \in \FOloc{p}(L_*)$ such that a $p$-tuple $\tpl{a}$ from $\str{A}*\str{G}$ satisfy $\psi$ if and only if $\tpl{a}$ has the profile $\pi$ and the representation of $\tpl{a}$ in $\str{A}$ satisfies $\phi$.
    To construct $\psi$, first check whether the profile of $\tpl{a}$ matches $\pi$ (using the relations $\Int^{\str{A}*\str{G}}$, $\Ext^{\str{A}*\str{G}}$, and $\rho^{\str{A}*\str{G}}$), then obtain the representation of each vertex $a_i$ (if $a_i \in \Int^{\str{A}*\str{G}}$, use $a_i$; otherwise, find $\rho(a_i)$ via the relation $\rho^{\str{A}*\str{G}}$), and use it to evaluate the formula $\phi$ with quantifiers restricted to $\Int^{\str{A}*\str{G}}$.
    This indeed does correspond to the evaluation of $\phi$ in $\str{A}$ with the representation of $\tpl{a}$ as arguments.
    Thus, we have
    \[
        \stonepar{\psi}{\str{A}*\str{G}} = \stonepar{\phi|_\pi}{\str{A}} \cdot \Pr[\tpl{a} \in V(\str{A}*\str{G}) \text{ has profile } \pi]
        .
    \]
    The probabilities $\stonepar{\psi}{\strseq{A}*\strseq{G}}$ converge by the assumption.
    Moreover, we can express by a local formula in the language $L_*$ whether a given $p$-tuple has profile $\pi$.
    Therefore, the probabilities $\Pr[\tpl{x} \text{ has profile } \pi]$ converge and, additionally, as $\pi$ is non-trivial, the limit is positive.
    It follows that the value $\stonepar{\phi|_\pi}{\str{A}}$ converges as well.
\end{proof}

Note that if a random vertex from $\str{A}*\str{G}$ is close to an $R$-edge (i.e. in a fixed finite distance) with positive limit probability, the same technique shows that also the gadgets are $\FOloc{m}$-convergent and, in fact, $\FO_m$-convergent for the appropriate $m$ from the statement of Theorem~\ref{thm:general_local_convergence}:
we use that the gadget $\str{G}$ can be interpreted from the structure $\str{H}$ induced from $\str{A}*\str{G}$ by an $R$-edge $\tpl{a}$ together with the set $\{b: \rho(b) = \tpl{a}\}$.
The additional elementary convergence follows from the fact that the diameter of $\str{H}$ is $2$ (all external vertices are connected to the $R$-edge); thus, local formulas in $\str{H}$ are able to test arbitrary sentences in $\str{G}$.
However, it is possible that a random vertex from $\str{A}*\str{G}$ is far from all $R$-edges a.a.s. and then we cannot say anything about the gadgets.

Naturally, it is possible to readily generalize the same technique to prove an inverse statement for Theorem~\ref{thm:local_multiple_gadgets} about multiple gadgets.

\subsection{Removing locality of construction language}\label{ssec:omitting_locality_of_construction_language}

In this section, we remove the feature of locality of gadget copies in structure $\str{A}*\str{G}$.
Here, we consider the structures $\str{A}*\str{G}$ to be $L_*$-structures, however, we redefine the relation $\rho^{\str{A}*\str{G}}$ so that it covers only the neighborhood of gadgets' roots.
That is,
\begin{multline*}
    \rho^{\str{A}*\str{G}} = \big\{([(\tpl{e}, v)], [e_1], \dots, [e_\arity{R}]) \;\big|\;
    \\
    [(\tpl{e}, v)] \in \Ext^{\str{A}*\str{G}} \text{ and } \exists i : \dist([(\tpl{e}, v)], [e_i]) = 1\big\}
    ,
\end{multline*}
where we measure the distance $\dist_{\str{A}*\str{G}}(v, \tpl{e})$ in the structure $\str{A}*\str{G}$ without $\rho^{\str{A}*\str{G}}$.

This change reveals importance of the fact whether the sequence of roots is negligible.

\begin{theorem}\label{thm:inverse_for_construction_language_without_locality}
    Suppose that $\strseq{A}*\strseq{G}$, with the modification above, is an $\FOloc{}(L_*)$-convergent sequence.
    Then either the conditional probabilities in sense of Theorem~\ref{thm:general_local_convergence} converge or the sequence $\tpl{z}^\strseq{G}$ is negligible in $\strseq{G}$.
\end{theorem}
\begin{proof}
    The idea and its execution is very similar to Theorem~\ref{thm:inverse_for_construction_language}.
    Let us assume that the sequence $\tpl{z}^\strseq{G}$ is not negligible.
    Therefore, there exists $r \in \N$ such that $\lim \nu_\strseq{G} N^r(\tpl{z}^{\strseq{G}}) > 0$.
    For a non-trivial $\pi$ and a formula $\phi \in \FOloc{}(L_R)$, we can create a formula $\psi$ that is satisfied by a tuple $\tpl{a}$ from $\str{A}*\str{G}$ if and only if the profile of $\tpl{a}$ is $\pi$, the representation of $\tpl{a}$ in $\str{A}$ satisfy $\phi$, \emph{and} the external vertices of $\tpl{a}$ are at distance at most $r$ from an internal vertex.
    As the limit probability of observing such a tuple $\tpl{a}$ is positive (using that $\pi$ is non-trivial and $\tpl{z}^{\strseq{G}}$ is not negligible), we proceed to the conclusion that the sequence $\stonepar{\phi|_\pi}{\str{A}}$ converges.
\end{proof}
\section{Applications}\label{sec:applications}

We present two simple applications of gadget construction.
We show that the question of $\FO$-convergence of a general sequence reduces to $\FO$-convergence of a sparse sequence, which has the property that asymptotically almost all $p$-tuples form an independent set.
Then for any given $k \geq 2$, we construct an almost surely $\FO_{k-1}$-convergent sequence of graphs which is not $\FO_k$-convergent.

\subsection{Reduction to sparse sequences}\label{ssec:sparse_sequences}

Here we prove that a sequence is $\FO$-convergent is and only if a certain sparse sequence is $\FO$-convergent.
More precisely, we show it for $\FO_p$-convergence.

\begin{proposition}\label{prop:reduction_to_sparse}
    Let $\strseq{A}$ be a sequence of $L$-structures.
    Denote by $\str{B}_n$ the structure $\str{A}_n$ with $n$ leaves marked by a distinct symbol attached to each vertex.
    The sequence $\strseq{A}$ is $\FO_p$-convergent if and only if the sequence $\strseq{B}$ is $\FO_p$-convergent.
\end{proposition}
\begin{proof}
    Attaching $n$ leaves to each vertex is a special case of gadget construction with unary $R$-edges.
    The sequence $\strseq{S}$ of stars on $n$ vertices with the center as the root is $\FO$-convergent.
    Thus, the implication from left to right is by Corollary~\ref{cor:preservation_of_elementary_convergence} and  Theorem~\ref{thm:general_local_convergence}.
    
    Conversely, if Spoiler has a winning strategy in $\EF_k(\str{A}_n;\str{A}_m)$, the same strategy surely works in $\EF_k(\str{B}_n;\str{B}_m)$ as Duplicator cannot use the new leaves due to their marks.
    Thus, if $\strseq{A}$ is not elementarily convergent, neither is $\strseq{B}$.
    As for the local convergence, note that in $\str{B}_n$ we can define all the relations from the construction language $L_*$ from Section~\ref{sec:inverse_theorems} (again, using the marks on the new vertices).
    Therefore, if $\strseq{B}$ is $\FOloc{p}(L)$-convergent, it is also $\FOloc{p}(L_*)$.
    Thus, the original sequence $\strseq{A}$ must also be $\FOloc{p}$-convergent by Theorem~\ref{thm:inverse_for_construction_language}:
    each non-trivial profile $\pi = (I,E_1, \dots, E_t)$ have $I = \emptyset$.
    However, selecting a random $R$-edge (representation of an external vertex) is the same as selecting a random vertex of $\str{A}_n$ thanks to $R^{\str{A}_n} = V(\str{A}_n)$ (abusing notation).
\end{proof}

The proposition implies that convergence of even very sparse structures is as complex as the general case.
This is in sharp contrast with e.g. the theory of left limits, which becomes trivial for graphs with a subquadratic number of edges.

\subsection{Graph sequences with bounded convergence}\label{ssec:bounded_convergence}

Here we use probability to construct an almost surely $\FO_{k-1}$-convergent sequence of graphs which is not $\FO_k$-convergent.

We write $H^k(n,p) = (V, E)$ for the random $k$-uniform hypergraph where each potential edge $X \in \binom{V}{k}$ belongs to $E$ with probability $p$.
We say that a $k$-uniform hypergraph $\str{H} = (V, E)$ has \emph{$q$-extension property} if for each $S \subseteq V$, $|S| = q-1$, and a partition $F_0, F_1$ of $\binom{S}{k-1}$ there is a vertex $v \in V \setminus S$ such that for $A \in \binom{S}{k-1}$ we have $A \cup \{v\} \in E$ if and only if $A \in F_1$.

Let $\strseq{H}$ be a sequence of $k$-uniform hypergraphs with $|V(\strseq{H})| \to \infty$.
Similarly to the case of graphs, if for each $q \in \N$ the hypergraphs from $\strseq{H}$ eventually have the $q$-extension property, then the sequence $\strseq{H}$ elementarily converges to the $k$-uniform Rado hypergraph $\mathcal{H}^k$.
For such sequences, $\FO$-convergence reduces to $\QF$-convergence, resp. $\FO_k$-convergence to $\QF_k$-convergence~\cite[Lemma~2.28]{unified_approach}, where $\QF$ is the set of quantifier-free formulas and $\QF_k$ the set of quantifier-free formulas with $k$ free variables.

\begin{example}\label{ex:bounded_convergence}
    Let $\strseq{H}$ be the following sequence of random $k$-uniform hypergraphs.
    \begin{align*}
        \str{H}_n =
        \begin{cases}
            H^k(n,p) & \text{if $n$ is odd}
            , \\
            H^k(n,q) & \text{if $n$ is even}
            ,
        \end{cases}
    \end{align*}
    where $0 < p < q < 1$.
    Such a sequence elementarily converges to $\mathcal{H}^k$ almost surely (similarly to~\cite[Lemma~2.33]{unified_approach}).
    Moreover, $\strseq{H}$ is $\QF_{k-1}$-convergent as each $(k-1)$-tuple of distinct vertices form an independent set.
    Obviously, $\strseq{H}$ is almost surely not $\QF_k$-convergent as $p \not= q$. 
    
    Put a unary $R$-edge to each vertex of $\str{H}_n$ and replace it by a gadget $\str{G}_n$, which is the star on $2^n$ vertices with the center as the root.
    Observe that the sequence $\strseq{H}' = \strseq{H} * \strseq{G}$ is (a.s.) $\FO_{k-1}$-convergent by Proposition~\ref{prop:reduction_to_sparse}.
    Then we replace each hyperedge by a gadget $\str{G}'_n$, which is the star on $k+1$ vertices with the leafs as the roots.
    The sequence $\strseq{H}'' = \strseq{H}' * \strseq{G}'$ of graphs is again $\FO_{k-1}$-convergent (a.s.) by Corollary~\ref{cor:preservation_of_elementary_convergence} and Theorem~\ref{thm:general_local_convergence} as the proportion of internal vertices (i.e. vertices of $\strseq{H}'$) tends to $1$.
    
    As a witness that $\strseq{H}''$ is almost surely not $\FO_k$-convergent, we can use the formula $\phi(\tpl{x}) \in \FO_k$ stating ``there is a vertex $y$ with $\dist(y, x_i) = 2$ for each $i$''.
\end{example}

We believe that the example illustrates what is, in some sense, the typical use of gadget construction.
That is, some constructions are simple when we are allowed to use edges of an arbitrary kind.
Using gadget construction, we can transfer the properties of the constructed objects to the more restricted graph setting.
\section{Conclusions and future work}\label{sec:conclusions}

In this paper, we have investigated the convergence of sequences created by gadget construction.
We hope that our results have shown gadget construction as a useful tool for creating convergent sequences of structures.

We believe that the natural step forward is to extend the results about convergence to limit structures.
Such results were obtained for the elementary limits, however, a general treatment for modelings is yet to be developed.

Finally, we pose a few open questions.
A positive answer to the first question would provide a large simplification to our presentation of fragmentation.

\begin{question}
    Fix an equivalence $\sigma$ on $[\arity{R}]$.
    Let $\strseq{A}^\sigma$ be an $\FO(L_\sigma)$-convergent sequence when the selection of random points is restricted to non-auxiliary vertices.
    Is there an $\FO$-convergent sequence $\strseq{B}$ with $\strseq{B}^\sigma \cong \strseq{A}^\sigma$, i.e. $\str{B}_n^\sigma \cong \str{A}_n^\sigma$ for each $n \in \N$?
    Is it true at least for elementary convergence?
\end{question}

The second question asks for an extension of Theorems~\ref{thm:inverse_for_construction_language}~and~\ref{thm:inverse_for_construction_language_without_locality}.

\begin{question}
    Is there a good description of those sequences $\strseq{A}$ and $\strseq{G}$ that produce a local convergent sequence $\strseq{A}*\strseq{G}$ with convergent proportion of internal vertices?
\end{question}

\printbibliography


\end{document}